\documentclass [11pt]{amsart}
\usepackage {amsmath, amssymb, amscd, mathrsfs, hyperref,  slashed, graphicx, color, enumerate, url}
\usepackage[text={6in,8.5in},centering,letterpaper,dvips]{geometry}
\usepackage[all,cmtip]{xy}

\newtheorem {theorem}{Theorem}[section]
\newtheorem {lemma}[theorem]{Lemma}
\newtheorem {proposition}[theorem]{Proposition}
\newtheorem {corollary}[theorem]{Corollary}

\newtheorem {definition}[theorem]{Definition}

\theoremstyle{remark}
\newtheorem {remark}[theorem]{Remark}
\newtheorem {example}[theorem]{Example}
\newtheorem {fact}[theorem]{Fact}

\DeclareFontFamily{U}{mathx}{\hyphenchar\font45}
\DeclareFontShape{U}{mathx}{m}{n}{
      <5> <6> <7> <8> <9> <10>
      <10.95> <12> <14.4> <17.28> <20.74> <24.88>
      mathx10
      }{}
\DeclareSymbolFont{mathx}{U}{mathx}{m}{n}
\DeclareFontSubstitution{U}{mathx}{m}{n}
\DeclareMathAccent{\widecheck}{0}{mathx}{"71}
\newcommand\eps{\varepsilon}
\def\polhk#1{\setbox0=\hbox{#1}{\ooalign{\hidewidth
    \lower1.5ex\hbox{`}\hidewidth\crcr\unhbox0}}}  
\def\Z {{\mathbb{Z}}}
\def\R {{\mathbb{R}}}
\def\C {{\mathbb{C}}}
\def\Q {{\mathbb{Q}}}
\def\Conf {\mathcal{C}}

\def\tR{\tilde{\R}}
\def\Spin {\mathbb{S}}
\def\S {\mathcal{S}}
\def\dirac {\slashed{\partial}}
\def\Dirac{\slashed{D}}
\def\tH{\tilde{H}}
\def\rp{\mathbb{RP}}

\def\hp{\mathbb{HP}}
\def\F {\mathbb{F}_2}
\def\del {\partial}
\def\tT{\tilde{T}}

\def\t{\mathfrak{t}}

\DeclareMathOperator{\id}{\operatorname{id}}
\def\ind{\operatorname{ind}}

\DeclareMathOperator{\im}{\operatorname{image}}

\def\To {\longrightarrow}

\def\hyp{\bigl( \begin{smallmatrix}0 & 1 \\ 1 & 0\end{smallmatrix} \bigr)}

\def\mod {\operatorname{mod}}

\def\swf{\operatorname{SWF}}
\def\swfh{\mathit{SWFH}}
\def\swfk{\mathit{SWFK}}
\def\iswfh{ ^{\infty}\! \swfh}
\def\pin {\operatorname{Pin}(2)}

\def\H {\mathbb{H}}
\def\iH{ ^{\infty}\! \tH}

\def\i {\mathfrak{I}}
\def\a {\mathfrak{a}}
\def\tK{\tilde{K}}
\def\tC{\tilde{\C}}
\def\tG{\tilde{G}}
\def\tc{\tilde{c}}
\def\tZ{\tilde{Z}}
\def\E{\mathfrak{E}}
\def\csd{\mathit{CSD}}

\newcommand{\s}{\mathfrak{s}}

\newcommand{\vnt}{V^\nu_\tau}
\def\vnu {V^{\nu}_{-\nu}}

\def\pt {\operatorname{pt}}

\def\eps{\varepsilon}
\begin{document}

\title[Intersection forms of spin four-manifolds with boundary]{On the intersection forms of spin four-manifolds with boundary}

\author[Ciprian Manolescu]{Ciprian Manolescu}
\thanks {The author was supported by NSF grant DMS-1104406.}
\address {Department of Mathematics, UCLA, 520 Portola Plaza\\ Los Angeles, CA 90095}
\email {cm@math.ucla.edu}

\begin{abstract}
We prove Furuta-type bounds for the intersection forms of spin cobordisms between homology $3$-spheres. The bounds are in terms of a new numerical invariant of homology spheres, obtained from $\pin$-equivariant Seiberg-\!Witten Floer K-theory. In the process we introduce the notion of a Floer $K_G$-split homology sphere; this concept may be useful in an approach to the 11/8 conjecture.
\end {abstract}

\maketitle

\section{Introduction}

Let $X$ be a smooth, oriented, spin $4$-dimensional manifold. Donaldson's diagonalizability theorem \cite{Donaldson, DonaldsonOr} implies that if $X$ is closed, then $X$ cannot have a non-trivial definite intersection form.  If $X$ is not closed but has boundary a homology $3$-sphere $Y$, its intersection form is still unimodular, and Fr{\o}yshov \cite{Froyshov} found constraints on the definite intersection forms of such $X$; see also \cite{FurutaBrazil, Nicolaescu}. These constraints depend on an invariant associated to the boundary $Y$ and, later, various other invariants of this type have been developed \cite{FroyshovYM, FroyshovHM, AbsGraded, KMOS, swfh}. 

With respect to indefinite forms, the situation is less understood. If $X$ is closed, Matsumoto's 11/8 conjecture \cite{Matsumoto} states that $b_2(X) \geq \frac{11}{8}|\sigma(X)|$. (Here, $\sigma$ denotes the signature.) Since $X$ is spin, its intersection form must be even. A unimodular, even indefinite form (of, say, nonpositive signature) can be decomposed as 
$$p(-E_8) \oplus q\hyp, \ \ p \geq 0, q > 0.$$
For forms coming from closed spin $4$-manifolds, Rokhlin's theorem \cite{Rokhlin} implies that $p$ is even. Since $b_2(X) = 8p+2q$ and $p=|\sigma(X)|/8$, the 11/8 conjecture can be rephrased as 
\begin{equation}
\label{eq:11_8}
q \geq  3p/2.
\end{equation}
An important result in this direction was obtained by Furuta \cite{Furuta}, who proved the inequality $b_2(X) \geq \frac{10}{8}|\sigma(X)| + 2$, i.e., 
\begin{equation}
\label{eq:10_8}
q \geq p + 1. 
\end{equation}
The free coefficient $1$ in the bound can sometimes be improved slightly, depending on the value of $p$ mod $8$; see \cite{Crabb, Stolz, Schmidt}.

The purpose of this paper is to obtain constraints on the indefinite intersection forms of spin four-manifolds with boundary. Although many of the results can be extended to the setting where the boundary $\del X = Y$ is a disjoint union of rational homology spheres (equipped with spin structures), for simplicity we will focus on the case where $Y$ consists of either one or two integral homology $3$-spheres. Then, the intersection form must still be of the type $p(-E_8) \oplus q\hyp$, and the parity of $p$ is given by the Rokhlin invariant of the boundary. (We allow here the case $p < 0$, and then we interpret $p(-E_8)$ as a direct sum of copies of the positive form $E_8$.) One method to obtain constraints on the intersection form is to pick a spin $4$-manifold $X'$ with boundary $-Y$, and apply Furuta's bound \eqref{eq:10_8} to the closed manifold $X \cup_{Y} X'$. This method can be refined by choosing $X'$ to be an orbifold rather than a manifold, and applying the orbifold version of \eqref{eq:10_8} developed by Fukumoto and Furuta in \cite{FukumotoFuruta}. We refer to \cite{BohrLee, FukumotoBG, Donald} for some results obtained using these methods.

In this paper we find further constraints by a different technique, based on adapting Furuta's proof of \eqref{eq:10_8} to the setting of manifolds with boundary. (However, our proof does not use the Adams operations, so it is in fact closer in spirit to Bryan's modification of Furuta's argument \cite{Bryan}.) Here is the first result:

\begin{theorem}
\label{thm:main1}
To every oriented homology $3$-sphere $Y$ we can associate an invariant $\kappa(Y) \in \Z$, with the following properties:
\begin{enumerate}[(i)]
\item The mod $2$ reduction of $\kappa(Y)$ is the Rokhlin invariant $\mu(Y)$;
\item Suppose that $W$ is a smooth, spin, negative-definite cobordism from $Y_0$ to $Y_1$, and let $b_2(W)$ denote the second Betti number of $W$. Then: 
$$\kappa(Y_1) \geq \kappa(Y_0) + \frac{1}{8} b_2(W).$$
\item Suppose that $W$ is a smooth, spin cobordism from $Y_0$ to $Y_1$, with intersection form $p(-E_8) \oplus q\hyp$. Then:
$$ \kappa(Y_1) + q \geq \kappa(Y_0) + p-1.$$
\end{enumerate}
\end{theorem}

Our main interest is in part (iii), but we listed properties (i) and (ii) in order to compare $\kappa$ with the invariants $\alpha, \beta, \gamma$ constructed in \cite{swfh}, using $\pin$-equivariant Seiberg-\!Witten Floer homology. The invariants $\alpha, \beta, \gamma$ satisfy the analogues of (i) and (ii). Property (iii) seems specific to the invariant $\kappa$, which is constructed from the $\pin$-equivariant Seiberg-\!Witten Floer K-theory of $Y$. The use of $\pin$-equivariant K-theory is to be expected, because it also appeared in Furuta's and Bryan's proofs of \eqref{eq:10_8}.

Roughly, the invariant $\kappa$ is defined as follows. We use the set-up from \cite{Spectrum, swfh}: Pick a metric $g$ on $Y$ and consider a finite dimensional approximation to the Seiberg-\!Witten equations, depending on an eigenvalue cut-off $\nu \gg 0$. The resulting flow has a Conley index $I_{\nu}$, which is a pointed topological space with an action by the group $G=\pin$. After changing $I_{\nu}$ by a suitable suspension if necessary, we can arrange for the $S^1$-fixed point set of $I_{\nu}$ to be equivalent to a complex representation sphere of $G$. We then consider the reduced $G$-equivariant K-theory of $I_{\nu}$. The inclusion of the $S^1$-fixed point set into $I_{\nu}$ induces a map 
\begin{equation}
\label{eq:iota}
\iota^*: \tK_G(I_{\nu}) \to \tK_G(I_{\nu}^{S^1}).
\end{equation}
We have a Bott isomorphism 
$$\tK_G(I_{\nu}^{S^1}) \cong K_G(\pt) = \Z[w, z]/(w^2-2w, wz-2w).$$
We let
$$ k(I_{\nu}) = \min \{k \geq 0 \mid \exists \ x \in \im(\iota^*) \subseteq K_G(\pt), wx = 2^k w \},$$
and obtain $\kappa(Y)$ from $2k(I_{\nu})$ by subtracting a correction term depending on $\nu$ and $g$.

The invariant $\kappa(Y)$ can be computed explicitly in some cases. For example:

\begin{theorem}
\label{thm:brieskorn}
$(a)$ We have $\kappa(S^3)=0$.

$(b)$ Consider the Brieskorn spheres $\Sigma(2,3,m)$ with $\gcd(m,6)=1$, oriented as boundaries of negative definite plumbings. Then:
\begin{align*}
\kappa(\Sigma(2,3,12n-1)) &= 2, \hskip1cm \kappa(\Sigma(2,3,12n-5)) = 1, \\
\kappa(\Sigma(2,3,12n+1)) &= 0, \hskip1cm \kappa(\Sigma(2,3,12n+5)) = 1.
\end{align*}

$(c)$ For the same Brieskorn spheres with the orientations reversed, we have 
\begin{align*}
\kappa(-\Sigma(2,3,12n-1)) &= 0, \hskip1cm \kappa(-\Sigma(2,3,12n-5)) = 1, \\
\kappa(-\Sigma(2,3,12n+1)) &= 0, \hskip1cm \kappa(-\Sigma(2,3,12n+5)) = -1.
\end{align*}
\end{theorem}

Observe that $\kappa(-Y)$ is not determined by $\kappa(Y)$. However, we can prove that $\kappa(Y) + \kappa(-Y) \geq 0$. (See Proposition~\ref{prop:duals}.) Furthermore, observe that for the examples appearing in Theorem~\ref{thm:brieskorn}, the values for $\kappa$ coincide with those for the invariant $\alpha$ defined in \cite{swfh}. We conjecture that $\kappa \neq \alpha$ in general; see Section~\ref{sec:alphas} for a discussion of this.

Note also that when $Y_0 = Y_1 = S^3$, the bound in Theorem~\ref{thm:main1} (iii) is weaker than Furuta's bound \eqref{eq:10_8}. We can remedy this by introducing the following concept:
\begin{definition}
\label{def:FloerSplit}
 We say that a homology sphere is {\em Floer $K_G$-split} if the image of the map $\iota^*$ from \eqref{eq:iota} is an ideal of $\Z[w, z]/(w^2-2w, wz-2w)$ of the form $(z^k)$ for some $k \geq 0$. 
\end{definition}

 For example, one can show that the three-sphere $S^3$, the Brieskorn spheres $\pm \Sigma(2,3,12n+1)$ and $\pm \Sigma(2,3,12n+5)$ are all Floer $K_G$-split, but the Brieskorn spheres of the form $\pm \Sigma(2,3,12n-1)$ and $\pm \Sigma(2,3,12n-5)$ are not Floer $K_G$-split. If the starting $3$-manifold in a cobordism $W$ is Floer $K_G$-split, we can strengthen the bound in Theorem~\ref{thm:main1} (iii):

\begin{theorem}
\label{thm:main2}
Suppose that $W$ is a smooth, spin cobordism from $Y_0$ to $Y_1$, with intersection form $p(-E_8) \oplus q\hyp$ and $q > 0$. If $Y_0$ is $K_G$-split, then:
$$\kappa(Y_1) + q \geq \kappa(Y_0) + p + 1.$$
\end{theorem}
Applying this to $Y_0= S^3$, which is Floer $K_G$-split, we obtain:

\begin{corollary}
\label{cor:1bdry}
Let $X$ be a smooth, compact, spin four-manifold with boundary a homology sphere $Y$. If the intersection form of $X$ is $p(-E_8) \oplus q\hyp$ and $q > 0$, then:
$$ q \geq p + 1- \kappa(Y).$$
\end{corollary}

When $Y=S^3$, we recover Furuta's $10/8$ Theorem \eqref{eq:10_8}. When $Y=\pm \Sigma(2,3,m)$ with $\gcd(m, 6)=1$, we get  specific bounds by combining Theorem~\ref{thm:brieskorn} (ii) and (iii) with Corollary~\ref{cor:1bdry}. In some of these cases, the bounds given by Corollary~\ref{cor:1bdry} can be obtained more easily by applying the orbifold version of Furuta's theorem to a filling of $X$. However, the bounds we get in the cases $Y=+\Sigma(2,3,12n+1)$ and $Y=+\Sigma(2,3,12n+5)$ appear to be new. We refer to Section~\ref{sec:bounds} for a detailed discussion.

The techniques developed in this paper may also be of interest in studying closed $4$-manifolds. Indeed, Bauer \cite{Bauer} proposed a strategy for proving the $11/8$ conjecture (in the simply connected case) by decomposing the $4$-manifold along homology spheres. Specifically, suppose we had a counterexample to \eqref{eq:11_8}, i.e., a closed, spin $4$-manifold $X$ with intersection form $2r(-E_8) \oplus q\hyp$ and $q < 3r$. By adding copies of $S^2 \times S^2$, we can assume that $q=3r-1$. If $\pi_1(X) =1$, then by a theorem of Freedman and Taylor \cite{FreedmanTaylor} we can find a decomposition 
\begin{equation}
\label{eq:bauer}
X = X_1 \cup_{Y_1} X_2 \cup_{Y_2} \dots \cup_{Y_{r-1}} X_r
\end{equation}
such that:
\begin{itemize}
\item $Y_i$ is an integral homology $3$-sphere for all $i$;
\item For $1 \leq i \leq r-1$, the manifold $X_i$ has intersection form $2(-E_8) \oplus 3\hyp$;
\item $X_r$ has intersection form $2(-E_8) \oplus 2\hyp$.
\end{itemize}
(There are several variations of this; e.g., one could ask for $X_1$ to have intersection form $2(-E_8) \oplus 4\hyp$ and for $X_r$ to have intersection for $2(-E_8) \oplus \hyp$, as in \cite{Bauer}.) If the homology spheres $Y_i$ are arbitrary, Theorem~\ref{thm:main1} (iii) is not sufficient to preclude the existence of such decompositions. On the other hand, Theorem~\ref{thm:main2} has the following immediate consequence: 

\begin{theorem}
\label{thm:cor}
There exists no closed four-manifold $X$ with a decomposition of the type \eqref{eq:bauer}, such that all the homology spheres $Y_i$ are Floer $K_G$-split.
\end{theorem}

In view of this result, it would be worthwhile to find topological conditions guaranteeing that a homology sphere is Floer $K_G$-split.

\medskip
\noindent {\bf Acknowledgements.} I would like to thank Mike Hopkins, Peter Kronheimer and Ron Stern for some very enlightening conversations, and the Simons Center for Geometry and Physics (where part of this work was written) for its hospitality. I am also grateful to Jianfeng Lin, Brendan Owens and the referee for comments on a previous version of this paper.

Some of the results in this article have been obtained independently by Mikio Furuta and Tian-Jun Li \cite{FurutaLi}.

\section{Equivariant K-theory}
\label{sec:eqK}

\subsection{Background}
We start by reviewing a few general facts about equivariant K-theory, mostly collected from \cite{Segal}; see also \cite{AtiyahBP}. We assume familiarity with ordinary K-theory, as in \cite{AtiyahK}.

Let $G$ be a compact topological group and $X$ a compact $G$-space. The equivariant K-theory of $X$, denoted $K_G(X)$, is the Grothendieck group associated to $G$-equivariant complex vector bundles on $X$. When $X$ is a point, $R(G) = K_G(\pt)$ is the representation ring of $G$. In general, $K_G(X)$ is an algebra over $R(G)$.

\begin{fact}
\label{fact:zero}
A continuous map $f:X \to X'$ induces a map $f^*:K_G(X') \to K_G(X)$. 
\end{fact}

\begin{fact}
\label{fact:res}
For every subgroup $H \subseteq G$, we have functorial restriction maps $K_G(X) \to K_H(X)$.
\end{fact}

\begin{fact}
\label{fact:free}
If $G$ acts freely on $X$, then the pull-back map $K(X/G) \to K_G(X)$ is a ring isomorphism.
\end{fact}

\begin{fact}
\label{fact:trivial}
If $G$ acts trivially on $X$, then the natural map $R(G) \otimes K(X) \to K_G(X)$ is an isomorphism of $R(G)$-algebras. 
\end{fact}

Now suppose that $X$ has a distinguished base point, fixed under $G$. We define the reduced equivariant $K$-theory of $X$, denoted $\tK_G(X)$, as the kernel of the restriction map $K_G(X) \to K_G(\pt)$. 

\begin{fact}
\label{fact:freebased}
If the action of $G$ on $X$ is free away from the basepoint, then the pull-back map $\tK(X/G) \to \tK_G(X)$ is a ring isomorphism.
\end{fact}

\begin{fact}
\label{fact:prod}
There is a natural product map $\tK_G(X) \otimes \tK_G(X') \to \tK_G(X \wedge X')$.
\end{fact}

If $V$ is any real representation of $G$, let $\Sigma^V X = V^+ \wedge X$ denote the (reduced) suspension of $X$ by $V$. When $V=n\R$ is a trivial representation, we simply write $\Sigma^n X$ for $\Sigma^{n\R} X$.

\begin{fact}
\label{fact:bott}
If $V$ is a complex representation of $G$, we have an equivariant Bott periodicity isomorphism, $\tK_G(X) \cong \tK_G(\Sigma^V X)$. This is given by multiplication with a Bott class $b_V \in \tK_G(V^+)$, under the natural map $ \tK_G(V^+) \otimes \tK_G(X) \to \tK_G(\Sigma^V X)$. The Bott isomorphism is functorial with respect to based continuous maps $f:X \to X'$.
\end{fact}

\begin{fact}
\label{fact:bott2}
Let $V$ be a complex representation of $G$. The composition of the Bott isomorphism with the map 
$\tK_G(\Sigma^V X) \to \tK_G(X)$ induced by the inclusion $X \hookrightarrow \Sigma^V X$ is a map $\tK_G(X) \to \tK_G(X)$ given by multiplication with the K-theoretic Euler class
$$ \lambda_{-1}(V) = \sum_d (-1)^d [\Lambda^i V] \in R(G).$$
\end{fact}

Bott periodicity for $V=\C \cong \R^2$ says that $\tK_G(\Sigma^{2} X) \cong \tK_G(X)$. For $i \in \Z$, we can define the reduced K-cohomology groups of $X$ by
$$ \tK_G^i(X) = \begin{cases}
\tK_G(X) & \text{if} \ i \text{ is even}, \\
\tK_G(\Sigma X) & \text{if} \ i \text{ is odd}.
\end{cases}$$

\begin{fact}
\label{fact:les}
If $A \subseteq X$ is a closed $G$-subspace (containing the base point), there is a long exact sequence:
\begin{equation}
\label{eq:pair}
 \to \tK_G^i(X \amalg_A CA) \to \tK^i_G(X) \to \tK^i_G(A) \to \tK_G^{i+1}(X \amalg_A CA) \to \dots 
 \end{equation}
 where $CA$ denotes the cone on $A$.
\end{fact}

A quick consequence of Fact~\ref{fact:les} is:
\begin{fact}
\label{fact:wedge}
If $X$ is a wedge sum $A \vee B$, then $\tK_G(X) \cong \tK_G(A) \oplus \tK_G(B)$.
\end{fact}
  
The {\em augmentation ideal} $\a \subseteq R(G)$ is defined as the kernel of the forgetful map (augmentation homomorphism) $R(G) \cong K_G(\pt)  \to K(\pt) \cong \Z$. The following fact is closely related to the Atiyah-Segal completion theorem; see \cite[proof of Proposition 4.3]{AtiyahSegal} and \cite[3.1.6]{AtiyahK}: 
 
 \begin{fact}
 \label{fact:complete}
If $X$ is a finite, based $G$-CW complex and the $G$-action is free away from the basepoint, then the elements of the augmentation ideal $\a \subset R(G)$ act nilpotently on $\tK_G(X) \cong \tK(X/G)$.
 \end{fact}

One can also define the equivariant $K$-groups when $X$ is only locally compact (see \cite{Segal}), e.g., for the classifying bundle $EG$. The following is a consequence of the Atiyah-Segal completion theorem; see \cite[Proposition 4.3]{AtiyahSegal} or \cite[Section XIV.5]{MayBook}: 

\begin{fact}
\label{fact:EG}
The ring $K_G(EG) \cong K(BG)$ is isomorphic to $R(G)^{\wedge}_{\a}$, the completion of $R(G)$ at the augmentation ideal. The projection $EG \to \pt$ induces a map $K_G(\pt)\to K_G(EG)$, which corresponds to the natural map from $R(G)$ to its completion.
\end{fact}

The following is an immediate corollary of Fact~\ref{fact:EG}:

\begin{fact}
\label{fact:free2}
Let $X$ be a compact space with a free $G$-action. Let $Q = X/G$ and denote by $\pi$ the projection $ X \to \pt$. The induced map $\pi^*$ from $R(G) \cong K_G(\pt)$ to $K(Q) \cong K_G(X)$ can also be described as the composition 
$$ R(G) \to R(G)^{\wedge}_{\a} \cong K(BG) \to K(Q),$$
where the first map is completion and the second is induced by the classifying map $Q \to BG$ for $X$.   
\end{fact}

\subsection{Pin(2)-equivariant K-theory}
From now on we specialize to the group $G = \pin$. If $\H = \C \oplus j \C$ denotes the algebra of quaternions, recall that $\pin$ can be defined as $S^1 \cup jS^1 \subset \H$. There is a short exact sequence
$$
 1 \To S^1 \To G \To \Z/2 \To 1.
$$

As in \cite{swfh}, we introduce notation for the following real representations of $G$:
\begin{itemize}
\item the trivial representation $\R$;
\item the one-dimensional sign representation $\tR$ on which $S^1 \subset G$ acts trivially and $j$ acts by multiplication by $-1$;
\item the quaternions $\H$, acted on by $G$ via left multiplication.
\end{itemize}

We also denote by $\tC$ the complexification $\tR \otimes_{\R} \C$; this is isomorphic to $\tR^2$ as a real representation.

\begin{fact}[\cite{Furuta, Schmidt}]
The representation ring $R(G)$ of $G=\pin$ is generated by $\tc = [\tC]$ and $h = [\H]$, subject to the relations $\tc^2 = 1$ and $\tc h = h$. 
\end{fact}

It will be convenient to use the generators:
$$w=\lambda_{-1}(\tC)= 1-\tc, \ \ \ \ \  z=\lambda_{-1}(\H)=2-h.$$ We obtain:
$$R(G) = \Z[w, z]/(w^2-2w, zw-2w).$$
The augmentation homomorphism is
\begin{equation}
\label{eq:aug}
R(G) \to \Z, \ \ \ w, z \mapsto 0.
\end{equation}
Therefore, the augmentation ideal of $R(G)$ is $\a = (w, z)$.

Observe also that restriction to the subgroup $S^1 \subset G$ induces the map
\begin{align}
\label{eq:rs1}
R(G) &\to R(S^1) = \Z[\theta, \theta^{-1}], \\ \notag
 w &\mapsto 0, \\ \notag
  z &\mapsto 2-(\theta + \theta^{-1}),
\end{align}
where $\theta$ is the class of the standard one-dimensional representation of $S^1$.

\section{The equivariant K-theory of spaces of type SWF} \label{sec:spacesSWF} 

In \cite[Section 2.3]{swfh} we defined a class of topological spaces with a $\pin$-action, called spaces of type $\swf$. These appear naturally in the context of finite dimensional approximation in Seiberg-\!Witten Floer theory; see Section~\ref{sec:swf} below. In \cite{swfh}, we found three numerical quantities (denoted $a$, $b$, and $c$) coming from the $\pin$-equivariant homology of a space of type $\swf$. Our goal in this section is to extract another quantity, denoted $k$, from the $\pin$-equivariant K-theory of a space of type $\swf$.

\subsection{A numerical invariant}
\label{sec:k}
We recall the following definition from \cite[Section 2.3]{swfh}:

\begin{definition}
Let $s \geq 0$. A {\em space of type $\swf$ (at level $s$)} is a pointed, finite $G$-CW complex $X$ with the following properties:
\begin{enumerate}[(a)]
\item The $S^1$-fixed point set $X^{S^1}$ is $G$-homotopy equivalent to the sphere $(\tR^{s})^+$;
\item The action of $G$ is free on the complement $X - X^{S^1}$.
\end{enumerate}
\end{definition}

We shall focus our attention on spaces of type $\swf$ at an {\em even} level. This is because if $s=2t$ is even, the $S^1$-fixed point set of $X$ is $G$-equivalent to the complex representation sphere $(\tC^t)^+$, so we can use equivariant Bott periodicity (Fact~\ref{fact:bott}) to get
$$ \tK_G(X^{S^1}) \cong \tK_G(S^0) \cong R(G).$$

We let $\iota: X^{S^1} \to X$ denote the inclusion, and let $\i(X)$ be the ideal of $R(G)$ with the property that the image of the induced map $\iota^*: \tK_G(X) \to \tK_G(X^{S^1})$ is $\i(X) \cdot b_{t\tC}$, where $b_{t\tC}$ is the Bott class.

\begin{lemma}
\label{lem:ii}
For any space $X$ of type $\swf$ at an even level, there exists $k \geq 0$ such that $w^k \in \i(X)$ and  $z^k \in \i(X)$.
\end{lemma} 

\begin{proof}
Apply the long exact sequence \eqref{eq:pair} to $A= X^{S^1} \subseteq X$:
$$ \dots \to \tK_G(X) \xrightarrow{\iota^*} \tK_G(X^{S^1}) \to \tK_G^1(X/X^{S^1}) \to \dots$$
 By the definition of spaces of type $\swf$, we know that $X/X^{S^1}$ has a free $G$-action away from the basepoint. By Fact~\ref{fact:complete}, we know that the elements $z, w \in \a$ act nilpotently on $\tK_G^1(X/X^{S^1}) \cong \tK^1((X/X^{S^1})/G)$. 

If $k$ is such that $z^k$ and $w^k$ act by $0$ on $\tK_G^1(X/X^{S^1})$, then the exact sequence implies that $z^k$ and  $w^k$ are in $\i(X)$.
\end{proof}

In $R(G)$ we have $w^2=wz=2w$, so $w \cdot w^k = w \cdot z^k = 2^k w$. In light of Lemma~\ref{lem:ii}, we can make the following:

\begin{definition}
\label{def:kx}
Given a space $X$ of type $\swf$ at an even level, we let
$$k(X) =  \min \{k \geq 0 \mid \exists \ x \in \i(X), \ w x=2^k w\}.$$
\end{definition}

Let us understand how the quantity $k$ behaves under suspensions. Note that if $X$ is a space of type $\swf$ at level $2t$, then $\Sigma^{\H}X$ is of type $\swf$ at the same level, and $\Sigma^{\tC} X$ is of type $\swf$ at level $2t+2$.

\begin{lemma}
\label{lem:susp}
If $X$ is a space of type $\swf$ at an even level, then
$$ \i(\Sigma^{\tC}X) = \i(X), \ \ \ \i(\Sigma^{\H} X) = z \cdot \i(X),$$
and consequently
$$ k(\Sigma^{\tC}X) = k(X), \ \ \ k(\Sigma^{\H} X) = k(X) + 1.$$
\end{lemma}

\begin{proof}
The statements about $\Sigma^{\tC}X$ follow from the fact that $(\Sigma^{\tC} X)^{S^1} = \Sigma^{\tC}(X^{S^1})$, together with the functoriality of the Bott isomorphism (Fact~\ref{fact:bott}).

To get the statements about $\Sigma^{\H}X$, note that inclusions of subspaces produce a commutative diagram
$$\begin{CD}
\tK_G(\Sigma^{\H}X) @>>> \tK_G(X) \\
@V{\iota_2^*}VV @VV{\iota_1^*}V \\
\tK_G((\Sigma^{\H} X)^{S^1}) @>{\cong}>> \tK_G(X^{S^1}).
\end{CD}$$
Since $(\Sigma^{\H} X)^{S^1} = X^{S^1}$, the bottom horizontal map is just the identity. Under the Bott isomorphism identification $\tK_G(\Sigma^{\H}X)\cong \tK_G(X)$, the top horizontal map is multiplication by $\lambda_{-1}(\H) = z$. This implies that
$$ \i(\Sigma^{\H}X)= z \cdot \i(X) \subseteq R(G).$$ 

If we are given $x \in \i(X)$ with $wx=2^kw$, we get $w(zx) = 2wx = 2^{k+1}w$. Conversely, if we have $x \in \i(X)$ with $w(zx) = 2^kw$, then $2wx = 2^kw$, so $wx = 2^{k-1}w$. Therefore, $k(\Sigma^{\H}X) = k(X) +1$.
\end{proof}

\subsection{Examples}
 \label{sec:exk}
 The simplest example of a space of type $\swf$ is $S^0$, for which $\i(S^0)=(1)$ and $ k(S^0)= 0.$ 
From Lemma~\ref{lem:susp} we get that
$$ \i((\tC^{t} \oplus \H^l)^+) = (z^l) \ \ \text{and} \ \ k( (\tC^t \oplus \H^l)^+) = l. $$

Further, observe that if $X$ is a space of type $\swf$ at level $2t$, and $X'$ is a space with a free $G$-action away from the basepoint, then the wedge sum $X \vee X'$ is also of type $\swf$ at level $2t$. The term $\tK_G(X')$ in $\tK_G(X \vee X') \cong \tK_G(X) \oplus \tK_G(X')$ does not interact with the $S^1$-fixed point set through the map $\iota^*$. Therefore,
\begin{equation}
\label{eq:kwedge}
 \i(X \vee X') = \i(X) \ \ \text{and} \ \ k(X \vee X') = k(X).
 \end{equation}
  
Combining the observations above, we find that if a space $X$ decomposes as a wedge sum of a representation sphere $(\tC^t \oplus \H^k)^+$ and a free $G$-space, then $\i(X)$ is of the form $(z^k)$.
We call such spaces {\em split}. 

We now introduce the following notion, which will play an important role in the paper:
 
\begin{definition}
\label{def:ksplit}
A space of type $\swf$ at an even level is called {\em $K_G$-split} if the ideal $\i(X) \subseteq R(G)$ is of the form $(z^k)$ for some $k \geq 0$.
\end{definition}

Thus, the $K_G$-split spaces are those that are indistinguishable from split spaces in terms of the ideal $\i(X)$. 

Let us give two examples of spaces of type $\swf$ that are not $K_G$-split. These examples were also considered in \cite[Section 2.4]{swfh}, and arise from the following construction. Suppose that $G$ acts freely on a finite $G$-CW-complex $Z$, and let $Q = Z/G$ be the respective quotient. Let $$\tilde Z = \bigl( [0,1]\times Z \bigr ) / (0,z) \sim (0, z') \text{ and } (1,z) \sim (1, z') \text{ for all } z,z' \in Z$$ denote the unreduced suspension of $Z$, where $G$ acts trivially on the $[0,1]$ factor. We view $\tilde Z$ as a pointed $G$-space, with one of the two cone points being the basepoint. Then $\tilde Z$ is of type $\swf$ at level $0$. There is a long exact sequence:
$$ \dots \To \tK_G(\tilde Z) \To \tK_G(S^0) \To \tK^{1}_G(\Sigma Z_+) \To \dots$$
Because $G$ acts freely on $Z$, we have $\tK^{1}_G(\Sigma Z_+) \cong \tK_G(Z_+) \cong K(Q)$, and the above sequence can be written
\begin{equation}
\label{eq:z}
0 \To K^1(Q) \To \tK_G(\tilde Z) \xrightarrow{\phantom{b}\iota^*} R(G) \xrightarrow{\phantom{b}\pi^*} K(Q) \To \tK^1_G(\tilde Z) \To 0.
\end{equation}
Here, the map $\pi^*$ is the one described in Fact~\ref{fact:free2}. Exactness tells us that the ideal $\i(\tZ)$ is the kernel of $\pi^*$.

\begin{example}
\label{ex:G}
Take $Z = G$, acting on itself via left multiplication, so that the quotient $Q$ is a single point. In the exact sequence \eqref{eq:z}, the map $\pi^*$ is the augmentation homomorphism $R(G) \to \Z$ from \eqref{eq:aug}. Therefore,
$$\i(\tilde G) = (w, z) \ \ \text{and} \ \ k(\tilde G) =  1.$$ 
Further, since $K^1(Q) =0$, we can compute
$$ \tK_G(\tilde G) \cong (w,z) \ \ \text{and} \ \ \tK^1_G(\tilde G)=0.$$
\end{example}

\begin{example}
\label{ex:torus}
Now let $Z$ be the torus
$$ T = S^1 \times jS^1 \subset \C \oplus j\C = \H,$$
with the $G$-action coming from $\H$. The quotient is $Q \cong S^1$. Since the inclusion of a point in $S^1$ induces an isomorphism 
$$K(S^1) \xrightarrow{\cong} K(pt) = \Z,$$
we deduce that the ideal $\i(\tilde T)$ is the same as in the previous example, that is,
\begin{equation}
\label{eq:zpn}
\i(\tilde T) = (w, z) \ \ \text{and} \ \ k(\tilde T) =  1.
 \end{equation}
 Further, from \eqref{eq:z}, we get that 
 $  \tK_G^1(\tilde T) = 0$ and $\tK_G(\tilde T)$ is an extension of $(w,z)$ by $K^1(Q) \cong \Z$. (Here, $\Z$ is $R(G)/(w,z)$ as an $R(G)$-module.)
\end{example}

\subsection{Properties}
We now turn to some general properties of the invariant $k(X)$.

\begin{lemma}
\label{lem:map0}
Let $X$ and $X'$ be spaces of type $\swf$ at even levels. Suppose that there exists a based, $G$-equivariant homotopy equivalence from $\Sigma^{r \R} X$ to $\Sigma^{r\R} X'$, for some $r \geq 0$. Then, we have $\i(X) = \i(X')$ and $k(X) = k(X')$.
\end{lemma}

\begin{proof}
By suspending the $G$-equivalence $\Sigma^{r\R} X \to \Sigma^{r\R} X'$ with another copy of $\R$ we can assume that $r$ is even, so that $U:=r \R = (r/2) \C$ is a complex representation. Consider the commutative diagrams
$$\xymatrix{
\tK_G(X') \ar[r]^{\cong} \ar[d] & \tK_G(\Sigma^U X')  \ar[r]^{\cong} \ar[d] & \tK_G(\Sigma^U X) \ar[r]^{\cong} \ar[d] & \tK_G(X) \ar[d] \\
\tK_G\bigl ((X')^{S^1} \bigr) \ar[r]^{\cong} & \tK_G\bigl( (\Sigma^U X')^{S^1} \bigr)  \ar[r]^{\cong} & \tK_G\bigl( (\Sigma^U X)^{S^1} \bigr) \ar[r]^{\cong} & \tK_G\bigl( (X)^{S^1} \bigr).
}$$
Here, in each row, the first map is a Bott isomorphism, the second comes from the $G$-equivalence in the hypothesis, and the third is the inverse to a Bott isomorphism. The vertical arrows are given by restriction. Comparing the first vertical arrow to the last we obtain the desired conclusions.
\end{proof}

\begin{lemma}
\label{lem:map1}
Let $X$ and $X'$ be spaces of type $\swf$ at the same even level $2t$, and suppose that $f: X \to X'$ is a $G$-equivariant map whose $S^1$-fixed point set map is an $G$-homotopy equivalence. Then:
$$ k(X) \leq k(X').$$
\end{lemma}

\begin{proof}
Analyzing the commutative diagram
$$\begin{CD}
\tK_G(X') @>{f^*}>> \tK_G(X) \\
@V{(\iota')^*}VV @VV{\iota^*}V \\
\tK_G((X')^{S^1}) @>{\cong}>> \tK_G(X^{S^1}),
\end{CD}$$
we see that $\i(X') \subseteq \i(X)$. This implies $k(X) \leq k(X')$.
\end{proof}

\begin{lemma}
\label{lem:map2}
Let $X$ and $X'$ be spaces of type $\swf$ at levels $2t$ and $2t'$, respectively, such that $t < t'$. Suppose that $f: X \to X'$ is a $G$-equivariant map whose $G$-fixed point set map is a homotopy equivalence. Then:
$$ k(X) +t \leq  k(X') + t'.$$
\end{lemma}

\begin{proof}
Note that the $G$-fixed point set of a space of type $\swf$ is homotopy equivalent to $S^0$. We have commutative diagrams:
\begin{equation}
\label{eq:2cd}
\begin{CD}
\tK_G(X') @>{f^*}>> \tK_G(X) \\
@V{(\iota')^*}VV @VV{\iota^*}V \\
\tK_G((X')^{S^1}) @>{(f^{S^1})^*}>> \tK_G(X^{S^1})\\
@V{\cdot w^{t'}}VV @VV{\cdot w^t}V \\
\tK_G((X')^G) @>{\cdot 1}>> \tK_G(X^G).
\end{CD}
\end{equation}
The bottom four groups are all isomorphic to $R(G)$, and we identified three of the maps with multiplications by elements of $R(G)$ using these isomorphisms; see Facts~\ref{fact:bott} and \ref{fact:bott2}. We deduce that the middle horizontal map $(f^{S^1})^*$ in \eqref{eq:2cd} is multiplication by an element $y \in R(G)$ such that 
\begin{equation}
\label{eq:y}
w^t \cdot y = w^{t'}.
\end{equation}

On the other hand, since $t < t'$, in view of Fact~\ref{fact:bott}, the map
$$ (f^{S^1})^*: \tK_{S^1}((X')^{S^1}) \to \tK_{S^1}(X^{S^1})$$
is zero. If we apply restriction maps $\tK_G \to \tK_{S^1}$ to the middle row in \eqref{eq:2cd} (compare Fact~\ref{fact:res}), we see that $y$ must be mapped to $0$ under the map $R(G) \to R(S^1)$ given by \eqref{eq:rs1}. This implies that $y=cw$ for some $c \in \Z$, and from \eqref{eq:y} we get 
$$ 2^t c w = c w^{t+1} = w^{t'} = 2^{t'-1}w,$$
so $c = 2^{t'-t-1}$. We deduce that $(f^{S^1})^*$ is multiplication by $2^{t'-t-1}w$.

Let $x \in \i(X')$ be such that $wx = 2^{k'} w$, where $k'=k(X')$. From the top diagram in \eqref{eq:2cd} we see that 
$$(f^{S^1})^*(x) = 2^{t'-t-1}wx = 2^{k'+t'-t-1}w \in \i(X).$$
Since $w(2^{k'+t'-t-1}w) = 2^{k'+t'-t}w$, we get that $k(X) \leq k'+t'-t$. 
\end{proof}

In the presence of the $K_G$-split assumption, we can strengthen Lemma~\ref{lem:map2}:
\begin{lemma}
\label{lem:map3}
Let $X$ and $X'$ be spaces of type $\swf$ at levels $2t$ and $2t'$, respectively, such that $t < t'$ and $X$ is $K_G$-split. Suppose that $f: X \to X'$ is a $G$-equivariant map whose $G$-fixed point set map is a homotopy equivalence. Then:
$$ k(X) + t + 1 \leq k(X')+t'.$$
\end{lemma}

\begin{proof}
Since $X$ is $K_G$-split, we must have $\i(X) = (z^k)$ where $k = k(X)$. The only modification is now in the last step of the proof of Lemma~\ref{lem:map2}. We know that $ 2^{k'+t'-t-1}w \in \i(X)$. An arbitrary element of $\i(X)$ is of the form $z^k(\lambda w + P(z))= \lambda 2^{k}w + z^kP(z),$ where $\lambda \in \Z$ and $P$  is a polynomial in $z$. For such an element to be a multiple of $w$ we must have $P(z)=0$. We get that $2^{k'+t'-t-1}w = \lambda 2^kw$, and hence $k' + t'-t-1 \geq k$.
\end{proof}

Finally, we mention the behavior of $k$ under equivariant Spanier-Whitehead duality. Let $V$ be a finite dimensional representation of $G$. Recall from \cite[Section XVI.8]{MayBook} that two pointed, finite $G$-spaces $X$ and $X'$ are {\em equivariantly $V$-dual} if there exist $G$-maps $\eps: X' \wedge X \to V^+$ and $\eta: V^+ \to X \wedge X'$ such that the following two diagrams are stably homotopy commutative:
$$\xymatrix{
V^+ \wedge X \ar[r]^{\eta \wedge \id} \ar[dr]^{\gamma} & X \wedge X' \wedge X  \ar[d]^{\id \wedge \eps}\\
& X \wedge V^+
}
 \ \  \ \  \ \ \ 
 \xymatrix{
 X' \wedge V^+ \ar[r]^{\id \wedge \eta} \ar[d]_{\gamma} & X' \wedge X \wedge X  \ar[d]^{\eps \wedge \id} \\
 V^+ \wedge X' \ar[r]_{r \wedge \id} & V^+ \wedge X',
 }
$$
where $r: V^+ \to V^+$ is the sign map, $r(v)=-v$, and $\gamma$ are the transpositions.

\begin{lemma}
\label{lem:dual}
Let $X$ and $X'$ be spaces of type $\swf$ at levels $2t$ resp. $2t'$, such that $X$ and $X'$ are equivariantly $V$-dual for some $V \cong \tC^s \oplus \H^l$, with $s, l \geq 0$. Then:
$$ k(X) + k(X') \geq l.$$
\end{lemma}

\begin{proof}
Consider the duality maps $\eps$ and $\eta$. Their restriction to the $S^1$-fixed point sets induce a $V^{S^1}$-duality between $X^{S^1} \simeq (\tC^t)^+$ and $(X')^{S^1} \simeq (\tC^{t'})^+$. Since $V^{S^1} \cong \tC^s$, this implies that $t+t'=s$. Let us view $\eps^{S^1}$ and $\eta^{S^1}$ as $G$-equivariant (that is, $\Z/2$-equivariant) maps from the sphere $(\tC^s)^+$ to itself. Their restrictions to the $\Z/2$-fixed point sets induce a duality between $S^0$ and $S^0$, that is, a bijection. This means that, up to $\Z/2$-equivalence, the maps $\eps^{S^1}$ and $\eta^{S^1}$ are unreduced suspensions of $\Z/2$-equivariant maps from the unit sphere $S(\tC^s)$ to itself. Up to $\Z/2$-equivalence, such maps $S(\tC^s) \to S(\tC^s)$ are determined by their degree (which must be odd by the Borsuk-Ulam theorem); see \cite{Olum}. We conclude that the maps $\eps^{S^1}$ and $\eta^{S^1}$ are determined (up to $G$-homotopy equivalences) by their degrees $d(\eps^{S^1}), d(\eta^{S^1}) \in \Z$. The duality diagrams imply that $d({\eps^{S^1}}) d(\eta^{S^1}) = \pm 1$, so $\eps^{S^1}$ and $\eta^{S^1}$ must be $G$-homotopy equivalences. Applying Lemma~\ref{lem:map1} to both of these maps we deduce that:
\begin{equation}
\label{eq:dualy} 
k(X \wedge X') = k(V^+)=l.
\end{equation}

Next, recall from Fact~\ref{fact:prod} that we have a product map $\tK_G(X) \otimes  \tK_G(X') \to \tK_G(X \wedge X')$. If $x \in \i(X)$ and $x' \in \i(X')$ are such that $wx=2^{k(X)}w$ and $wx'=2^{k(X')}w$, then $xx' \in \i(X)$ and 
$$ w(x x') = 2^{k(X)}wx'= 2^{k(X)+k(X')} w.$$
We deduce that:
\begin{equation}
\label{eq:smash}
 k(X \wedge X') \leq k(X) + k(X').
 \end{equation}
Combining this with \eqref{eq:dualy}, the conclusion follows.
\end{proof}

To see an example when the inequality~\eqref{eq:smash} is strict, let $X$ be the space $\tilde G$ from Example~\ref{ex:G}, and let $X'$ be the space $\tilde T$ from Example~\ref{ex:torus}. We showed that $k(\tilde G) = k(\tilde T)=1$, and it is observed in \cite[Example 2.14]{swfh} that $\tilde G$ and $\tilde T$ are $\H$-dual. From Equation~\eqref{eq:dualy} we get that $ k(\tilde G \wedge \tilde T) = 1 < k(\tilde G) + k(\tilde T) = 2.$ 

\section{Pin(2)-equivariant Seiberg-Witten Floer K-theory}
\label{sec:swf}

In this section we use the methods in \cite{Spectrum} and \cite{swfh} to construct $\pin$-equivariant Seiberg-\!Witten Floer K-theory. We will start by working in the setting of rational homology spheres (equipped with a spin structure), but when we discuss applications we will specialize to integral homology spheres.

\subsection{Finite dimensional approximation}
\label{sec:fda}
Let us briefly review the construction of equivariant Seiberg-\!Witten Floer spectra. We refer to \cite{Spectrum} and \cite{swfh} for more details.

Let $Y$ be a rational homology three-sphere, $g$ is a metric on $Y$, $\s$ a spin structure on $Y$, and $\Spin$ the spinor bundle for $\s$. Consider the global Coulomb slice in the Seiberg-\!Witten configuration space:
$$V = i \ker d^* \oplus \Gamma(\Spin) \subset i\Omega^1(Y) \oplus \Gamma(\Spin).$$ 

Using the quaternionic structure on spinors, we find an action of the group $G=\pin$ on $V$. Precisely,  an element $e^{i\theta} \in S^1$ takes $(a, \phi)$ to $(a, e^{i\theta}\phi)$, whereas $j \in G$ takes $(a, \phi)$ to $(-a, j\phi)$.

Let $\rho:TY \to \text{End}(\Spin)$ denote the Clifford multiplication, and $\dirac : \Gamma(\Spin) \to \Gamma(\Spin)$ the Dirac operator. The Chern-Simons-Dirac functional $\csd: \Conf(Y, \s) \to \R$, given by:
\[
\csd(a,\phi) = \frac{1}{2} \bigl(\int_Y \langle \phi, \dirac \phi + \rho(a)\phi \rangle dvol - \int_Y a \wedge da \bigr), 
\]
is invariant under the $G$-action. Its gradient (in a suitable metric) is the Seiberg-\!Witten map, which decomposes as a sum 
$$ \ell + c: V \to V,$$
where $\ell$ is the linearization $\ell(a,\phi) = (*da, \dirac \phi)$. We refer to the gradient flow of $\csd$ as the Seiberg-\!Witten flow.

The map $\ell$ is an elliptic, self-adjoint operator. We denote by $\vnt$ the finite-dimensional subspace of $V$ spanned by the eigenvectors of $\ell$ with eigenvalues in the interval $(\tau,\nu]$. Note that, as a $G$-representation, $\vnt$ decomposes as a direct some of some copies of $\tR$ and some copies of $\H$. We write this decomposition as
$$ \vnt = \vnt(\tR) \oplus \vnt(\H).$$

Next, we consider the gradient flow of the restriction $\csd|_{\vnt}$, where $\nu \geq 0$ and $\tau \ll 0$. We view this as a finite dimensional approximation to the Seiberg-\!Witten flow. The eigenvalue cut-offs $\nu$ and $\tau$ can be chosen independently. However, for simplicity, we shall restrict to the case $\tau = - \nu$. 

We pick $R \gg 0$ (independent of $\nu$) such that all the finite energy Seiberg-\!Witten flow lines are inside the ball $B(R)$ in a suitable Sobolev completion of $V$. We then look at the approximate Seiberg-\!Witten flow on $\vnu$. It can be shown that the points lying on trajectories of this flow that stay inside $B(R)$ form an isolated invariant set. To this set one can associate an equivariant Conley index $I_{\nu}$, which is a pointed $G$-space, well-defined up to canonical $G$-homotopy equivalence. (Roughly, one can think of the Conley index as the quotient of $\vnu \cap B(R)$ by the subset of $\vnu \cap \del B(R)$ where the flow exits the ball.) The following facts are established in \cite{Spectrum, swfh}:

\begin{proposition}
\label{prop:conley}
$(a)$ The Conley index $I_{\nu}$ is a space of type $\swf$ at level $\dim V^0_{-\nu}(\tR)$.

\noindent $(b)$ When we vary the choices in its construction, the Conley index $I_{\nu}$ changes as follows:
\begin{enumerate}[(i)]
\item When we vary the radius $R$, it only changes by a $G$-equivalence;
\item When we change the cut-off $\nu$ to some $\nu' > \nu$, the space $I_{\nu'}$ is $G$-equivalent to the suspension of $I_{\nu}$ by the representation $V^{-\nu}_{-\nu'}$;
\item If we vary the Riemannian metric $g$ by a small homotopy, we can choose a cut-off $\nu$ such that the operator $\ell$ does not have $\nu$ or $-\nu$ as an eigenvalue during the homotopy. Then $I_{\nu}$ only changes by a $G$-equivalence.  
\end{enumerate}
\end{proposition}

For a fixed metric $g$, we can build a universe made of the negative eigenspaces of $\ell$ (together with infinitely many copies of the trivial $G$-representation), and construct a spectrum $\swf(Y, \s, g)$ as the formal de-suspension $\Sigma^{-V^0_{-\nu}} I_{\nu}$; see \cite[Section 3.4]{swfh}. In view of properties (i) and (ii) in Proposition~\ref{prop:conley}(b), the spectrum $\swf(Y, \s, g)$ is independent of $R$ and $\nu$, up to $G$-equivalence. We call $\swf(Y, \s, g)$ the {\em Seiberg-\!Witten Floer spectrum} of the triple $(Y, \s, g)$.

When we vary the metric $g$, it is difficult to identify the universes that provide coordinates for our spectra. Note that, for fixed $\nu$, the dimension of $V^0_{-\nu}$ changes according to the spectral flow of the operator $\ell = *d \oplus \dirac$. The operator $*d$ has trivial spectral flow, but the Dirac operator has spectral flow given by the formula
\begin{equation}
\label{eq:spflow}
 \text{S.F.}(\dirac) = n(Y, \s, g_0) - n(Y, \s, g_1).
 \end{equation}
Here, $g_0$ and $g_1$ are the initial and final metrics, and the quantities 
$$n(Y, \s, g_i) \in \tfrac{1}{8} \Z \subset \Q$$ are linear combinations of the eta invariants associated to $*d$ and $\dirac$, for each metric. Alternatively, given a metric $g$ on $Y$, we can pick a compact spin 
four-manifold $W$ with boundary $Y$, let $\Dirac(W)$ be the Dirac operator on $W$ (with Atiyah-Patodi-Singer boundary conditions), and set
\begin{equation}
\label{eq:n}
 n(Y, \s, g) = \ind_{\C}\Dirac(W) + \frac{\sigma(W)}{8}.
 \end{equation}

Although $n(Y, \s, g)$ is in general one-eighth of an integer, as we vary $g$ (and keep $Y$ and $\s$ fixed) it changes by elements of $\Z$. Also, when $Y$ is an integral homology sphere, we have $n(Y, \s, g) \in \Z$, and its parity is given by the Rokhlin invariant $\mu$:
\begin{equation}
\label{eq:nmod2}
n(Y, \s, g) \mod 2 = \mu(Y) \in \Z/2.
\end{equation}

Looking at \eqref{eq:spflow}, one is prompted to consider a formal de-suspension of $\swf(Y, \s, g)$ by $n(Y, \s, g)/2$ copies of the representation $\H$. (The factor of $1/2$ comes from the fact that \eqref{eq:spflow} counts complex dimensions of the eigenspaces of $\dirac$, rather than quaternionic dimensions.) This produces an invariant of $Y$ in the form of an equivalence class of formally de-suspended spaces. The relevant definition is given below.

\subsection{Stable even equivalence}
Consider the set of triples $(X, m, n)$, where $X$ is a space of type $\swf$ at an even level, $m \in \Z$ and $n \in \Q$. We introduce the following equivalence relation on such triples:
\begin{definition}
\label{def:see}
We say that $(X, m,n)$ is {\em stably even equivalent} to $(X', m', n')$ if $n-n' \in \Z$, and there exist $M, N, r \geq 0$ and a $G$-homotopy equivalence 
$$ \Sigma^{r \R} \Sigma^{(M-m) \tC} \Sigma^{ (N-n) \H} X \xrightarrow{\sim} \Sigma^{r \R} \Sigma^{(M-m') \tC} \Sigma^{(N-n') \H} X.$$
\end{definition}
Thus, a triple could be thought of as a ``$G$-equivariant suspension spectrum,'' given by the formal de-suspension of $X$ by $m$ copies of the representation $\tC$ and $n$ copies of the representation $\H$.

We denote by $\E$ the set of stable even equivalence classes of triples $(X, m, n)$. Informally, we will refer to the elements of $\E$ as {\em spectrum classes}.

If $(X, m, n)$ is a triple as above, we define its (reduced) equivariant Borel cohomology, with coefficients in an Abelian group $A$, by
$$ \tH^*_G(X, m, n; A) := \tH_G^{*+2m+4n}(X; A)$$
and its (reduced) equivariant K-cohomology by
$$ \tK^*_G(X, m, n) := \tK_G^{*+2m+4n}(X).$$
We also set
$$k(X, m, n) = k(X)-n,$$
where $k$ is the invariant defined in Section~\ref{sec:k}.

\begin{lemma}
\label{lem:evens}
Let $(X, m, n)$ be a triple as above. Then, the following are invariants of the spectrum class $\S=[(X, m, n)] \in \E$: 
\begin{itemize}
\item The isomorphism class of Borel cohomology, $\tH^*_G(\S; A) := [\tH^*_G(X, m, n; A)]$, as a graded module over $H^*(BG; A)$; 
\item The isomorphism class of equivariant K-cohomology, $\tK_G^*(\S) := [\tK^*_G(X, m, n)]$, as a graded module over $R(G)$;
\item The quantity $k(\S) := k(X, m, n) \in \Q$.
\end{itemize}
\end{lemma}

\begin{proof}
The first two statements follow from the invariance of the two theories under suspensions by complex representations; compare \cite[Remark 2.3]{swfh} and \ref{fact:bott}.

The third statement follows from the behavior of $k$ under $G$-equivalences (after stabilization by copies of $\R$) and under suspensions by $\tC$ and $\H$. These were established in Lemma~\ref{lem:map0} and Lemma~\ref{lem:susp}, respectively. 
\end{proof}

\begin{remark} 
In the definition of stable even equivalence we only allowed de-suspensions by copies of $\tC = \tR \oplus \tR$ and $\H$, which are complex representations of $G$. We did this because equivariant cohomology and equivariant K-theory are invariant (up to a shift in degree) under such representations, whereas they are not invariant under suspending by an arbitrary real representation such as $\tR$. If we had been interested only in the equivariant cohomology with $\Z/2$ coefficients (as we were in \cite{swfh}), then we could have allowed de-suspensions by $\tR$, and dropped the condition on $X$ to be at an even level.

Note also the presence of arbitrary suspensions by $\R$ in Definition~\ref{def:see}. This is not  necessary for constructing a $3$-manifold invariant as a spectrum class (which we do in Section~\ref{sec:swfclass} below), but it makes computations more accessible. For example, when we compute some spectrum classes in Section~\ref{sec:brieskorn}, we will be free to use standard facts from equivariant stable homotopy.
\end{remark}

\subsection{The Seiberg-Witten Floer spectrum class}
\label{sec:swfclass}
If $Y$ is a rational homology sphere with a spin structure $\s$, let $g, R$ and $\nu$ be as in Subsection~\ref{sec:fda}. Recall from Proposition~\ref{prop:conley}(a) that the Conley index $I_{\nu}$ is a space of type $\swf$ at level $\dim V^0_{-\nu}(\tR)$. Define 
$$ \S(Y, \s) = 
\begin{cases} [(I_{\nu}, \tfrac{1}{2}\dim V(\tR)^0_{-\nu}, \dim_{\H} V(\H)^0_{-\nu} + \tfrac{1}{2}n(Y, \s, g)] & \text{if} \ I_{\nu} \text{ is at an even level,} \\ 
[(\Sigma^{\tR} I_{\nu},  \tfrac{1}{2}(\dim V(\tR)^0_{-\nu} + 1), \dim_{\H} V(\H)^0_{-\nu} + \tfrac{1}{2}n(Y, \s, g)] & \text{if} \ I_{\nu} \text{ is at an odd level.}
\end{cases}
$$

\begin{proposition}
\label{prop:inv}
The spectrum class $\S(Y, \s) \in \E$ is an invariant of the pair $(Y, \s)$.
\end{proposition}

\begin{proof}
This is a consequence of Proposition~\ref{prop:conley}(b) and the formula \eqref{eq:spflow} for the spectral flow. Compare \cite[Theorem 1]{Spectrum}.
\end{proof}

We refer to $\S(Y, \s)$ as the {\em Seiberg-\!Witten Floer spectrum class of $(Y, \s)$}. 

In view of Lemma~\ref{lem:evens} and Proposition~\ref{prop:inv}, we define the {\em $G$-equivariant Seiberg-\!Witten Floer cohomology} of $(Y, \s)$, with coefficients in an Abelian group $A$, as
$$ \swfh^*_G(Y, \s; A) := \tH^*_G(\S(Y, \s); A).$$
(One can also define equivariant Seiberg-\!Witten Floer homology in a similar manner.) Further, we define the {\em $G$-equivariant Seiberg-\!Witten Floer K-cohomology} of $(Y, \s)$ as
$$ \swfk^*_G(Y, \s) := \tK^*_G(\S(Y, \s)).$$
This group in degree $2n(Y, \s, g) \in \tfrac{1}{4}\Z$ can be called the {\em $G$-equivariant Seiberg-\!Witten Floer K-theory} of $(Y, \s)$. 

We define:
\begin{equation}
\label{eq:kappa}
 \kappa(Y, \s) := 2k(\S(Y, \s)) \in \tfrac{1}{8} \Z \subset \Q.
 \end{equation}
We say that the pair $(Y, \s)$ is {\em Floer $K_G$-split} if, for $\nu \gg 0$,  either $I_{\nu}$ or $\Sigma^{\tR} I_{\nu}$ (depending on the parity of the level of $I_{\nu}$) is $K_G$-split in the sense of Definition~\ref{def:ksplit}; cf. Definition~\ref{def:FloerSplit} from the introduction.  

If $Y$ is an integral homology sphere, then it has a unique spin structure $\s$, which we drop from the notation. Since $n(Y, g) \in \Z$, in this case $ \swfh^*_G(Y ; A)$ and $\swfk^*_G(Y)$ are integer-graded, and we have $ \kappa(Y) \in \Z$.

\subsection{Cobordisms}
Suppose $W$ is a four-dimensional, oriented cobordism between rational homology spheres $Y_0$ and $Y_1$, such that $b_1(W)=0$. Further, assume $W$ is equipped with a Riemannian metric $g$ and a spin structure $\t$. It is shown in \cite[Section 9]{Spectrum} and \cite[Section 3.6]{swfh} that one can do finite dimensional approximation for the Seiberg-\!Witten equations on $W$ to obtain a map:
\begin{equation}
\label{eq:cob}
f: \Sigma^{m_0 \tR} \Sigma^{n_0 \H} (I_0)_{\nu} \longrightarrow   \Sigma^{m_1 \tR} \Sigma^{n_1 \H} (I_1)_{\nu}.
\end{equation}
Here, $(I_0)_{\nu}$ and $(I_1)_{\nu}$ are the Conley indices for the approximate Seiberg-\!Witten flows on $Y_0$ and $Y_1$, respectively, corresponding to an eigenvalue cut-off $\nu \gg 0$. Let also $V_i$  denote the global Coulomb slice on $Y_i$, for $i=0,1$. The differences in suspension indices in \eqref{eq:cob} are:
$$ m_0 - m_1 = \dim_{\R} \bigl( (V_1)^0_{-\nu}(\tR)\bigr) - \dim_{\R} \bigl((V_0)^0_{-\nu}(\tR)\bigr) - b_2^+(W) $$ 
and
$$ n_0 - n_1 =   \dim_{\H} \bigl( (V_1)^0_{-\nu}(\H) \bigr) - \dim_{\H} \bigl( (V_0)^0_{-\nu}(\H) \bigr)  + n(Y_1, \t|_{Y_1}, g)/2 - n(Y_0, \t|_{Y_0}, g)/2 - {\sigma(W)}/16.$$
Moreover, the $S^1$-fixed point set of \eqref{eq:cob} is induced on the one-point compactifications by a linear injective map with cokernel of dimension $b_2^+(W)$.

Note that both the domain and the target of the map \eqref{eq:cob} are spaces of type $\swf$. The difference in their levels is $-b_2^+(W)$. If both levels happen to be even, then the difference in the values of $k$ for the domain and the target is
$$\frac{1}{2} \bigl( \kappa(Y_0) - \kappa(Y_1) - \sigma(W)/8\bigr).$$
 
We can now give the proofs of the main results advertised in the introduction:

\begin{proof}[Proof of Theorem~\ref{thm:main1}] Part (i) follows from the formula \eqref{eq:kappa} for $\kappa$, the definition of the spectrum class $\S(Y)$, and the fact that $n(Y, g)$ mod $2$ is the Rokhlin invariant; cf. \eqref{eq:nmod2}.

For part (ii), after doing surgery on loops, we can assume without loss of generality that $b_1(W)=0$. Consider the map \eqref{eq:cob} associated to the cobordism $W$. Since $b_2^+(W) =0$, 
the domain and target of \eqref{eq:cob} are at the same level. By suspending the map $f$ with $\tR$ if necessary, we can arrange that the common level is even. The conclusion then follows from Lemma~\ref{lem:map1}.

For part (iii), again we can assume that $b_1(W)=0$. If the intersection form on $W$ is $p(-E_8) \oplus q\hyp$, the difference in levels in \eqref{eq:cob} is $-q$. If $q=0$, we can simply apply part (ii). If $q > 0$ and $q$ is even, since $p=-\sigma(W)/8$, by applying Lemma~\ref{lem:map2} to \eqref{eq:cob} we get:
$$ \kappa(Y_0) + p \leq  \kappa(Y_1) + q.$$
If $q > 0$ and $q$ is odd, the best we can do is to take the connected sum of $W$ and a copy of $S^2 \times S^2$ to reduce to the case of $q$ even. We do this at the expense of weakening the bound above to: $\kappa(Y_0) + p -1 \leq  \kappa(Y_1) + q.$
\end{proof}

\begin{remark}
Parts (i) and (ii) of Theorem~\ref{thm:main1} admit straightforward generalizations to the case when $Y_0$ and $Y_1$ are rational homology spheres equipped with spin structures. There is also an analogue of part (iii)  which can be used to get constraints on the indefinite intersection forms of spin cobordisms between rational homology spheres; however, these intersection forms are not generally unimodular, so we cannot write them as $p(-E_8) \oplus q\hyp$. The bound in (iii) can be expressed instead in terms of the second Betti number and the signature of $X$. 
\end{remark}

\begin{proof}[Proof of Theorem~\ref{thm:main2}]
The same argument as in part (iii) of Theorem~\ref{thm:main1} applies here, except that now we can use Lemma~\ref{lem:map3} instead of Lemma~\ref{lem:map2}. When $q$ is even we get 
\begin{equation}
\label{eq:qeven}
 \kappa(Y_0) + p + 1 \leq  \kappa(Y_1) + q.
 \end{equation}
By part (i) of Theorem~\ref{thm:main1} the parity of $\kappa(Y_0)-\kappa(Y_1)$ is the Rokhlin invariant of the boundary of $W$, so it is the same as the parity of $p$. Therefore, for parity reasons we can improve the inequality \eqref{eq:qeven} to
$$ \kappa(Y_0) + p + 2 \leq  \kappa(Y_1) + q.$$
When $q$ is odd, we add a copy of $S^2 \times S^2$ and we are left with the inequality \eqref{eq:qeven}.
\end{proof}

\begin{proof}[Proof of Corollary~\ref{cor:1bdry}]
In Section~\ref{sec:psc} below we will prove that $\S(S^3) = [(S^0, 0, 0)]$, so $S^3$ is Floer $K_G$-split and $\kappa(S^3)=0$. Assuming this, we can apply Theorem~\ref{thm:main2} to the complement of a ball in $W$.
\end{proof}

\begin{proof}[Proof of Theorem~\ref{thm:cor}]
Suppose such a decomposition exists. Applying Corollary~\ref{cor:1bdry} to the first piece $X_1$ we get $\kappa(Y_1) \geq 2+1-3 = 0.$ Next, apply Theorem~\ref{thm:main2} to the pieces $X_i$ for $i=1, \dots, r-1$. We obtain $\kappa(Y_{i+1}) \geq \kappa(Y_{i})$ for all such $i$, so $\kappa(Y_{r-1}) \geq \kappa(Y_1) \geq 0$. On the other hand, by applying Theorem~\ref{thm:main2} to the complement of a ball in $X_r$ we get $\kappa(Y_{r-1}) \leq -1$, a contradiction.

A similar argument can be used to exclude any decompositions of $X$ into $r$ spin pieces, each of signature $-16$, and glued along Floer $K_G$-split homology spheres.
\end{proof}

Here is one last result mentioned in passing in the introduction:

\begin{proposition}
\label{prop:duals}
If $Y$ is an oriented homology sphere and $-Y$ is the same manifold with the reverse orientation, then:
$$ \kappa(Y) +\kappa(-Y) \geq 0.$$
\end{proposition}

\begin{proof}
It is shown in \cite[proof of Proposition 3.9]{swfh} that the Conley indices $I_{\nu}$ (for $Y$) and $\bar I_{\nu}$ (for $-Y$) are equivariantly $(V^{\nu}_{-\nu})$-dual to each other. The result now follows from  Lemma~\ref{lem:dual}.
\end{proof}

\section{Calculations}
\label{sec:calc}

In this section we prove Theorem~\ref{thm:brieskorn} from the introduction, about the values of $\kappa$ for $S^3$ and for the Brieskorn spheres $\pm \Sigma(2,3,m)$ with $\gcd(m, 6)=1$. We obtain some concrete bounds on the intersection forms of spin four-manifolds with boundary, and compare them to the bounds that can be obtained by simpler methods.

\subsection{Positive scalar curvature}
\label{sec:psc}
If $Y$ is a rational homology sphere admitting a metric $g$ of positive scalar curvature, by the arguments in \cite[Section 10]{Spectrum} or \cite[Section7.1]{GluingBF}, we obtain
$$\S(Y, \s) = [(S^0, 0, n(Y, \s, g)/2)]$$
and therefore
$$ \kappa(Y, \s) = - n(Y, \s, g).$$
In particular, $\S(S^3)=[(S^0, 0, 0)]$ and $\kappa(S^3)=0$. This proves part (i) of Theorem~\ref{thm:brieskorn}.

\subsection{A family of Brieskorn spheres}
\label{sec:brieskorn}
We now move to parts (ii) and (iii) of Theorem~\ref{thm:brieskorn}. 

We use the arguments in \cite[Section 7.2]{GluingBF} and \cite[Section 3.8]{swfh} to compute explicitly the Seiberg-\!Witten Floer spectrum classes of $\pm \Sigma(2,3,m)$. The calculations are based on the description of the monopole solutions on $\Sigma(2,3,m)$, which was given by Mrowka, Ozsv\'ath and Yu in \cite{MOY}.

We start with the case $m=12n-1$. The Seiberg-\!Witten equations on $\Sigma(2,3,12n-1)$ have one reducible solution in degree zero, and $2n$ irreducibles in degree one. The irreducibles come in $n$ pairs related by the action of the element $j \in G$. Thus, a representative for $\S(\Sigma(2,3,12n-1))$ can be constructed by attaching $n$ free cells of the form $\Sigma G_+$ to a trivial cell $S^0$. The attaching map for each cell is determined by a stable homotopy class in $\{G_+, S^0\}_G \cong \{S^0, S^0\} \cong \Z$. Together the attaching maps given an element in $\Z^n$, and the spectrum class is  determined by the divisibility of this element. The fact that it is primitive can be deduced from the calculation of the $S^1$-equivariant homology of $\S(\Sigma(2,3,12n-1))$, given in \cite[Section 7.2]{GluingBF}. (In fact, it even suffices to know the non-equivariant homology.) We obtain: 
$$ \S(\Sigma(2,3,12n-1)) = [(\tilde G \vee \underbrace{ \Sigma G_+ \vee \dots \vee \Sigma G_+}_{n-1}, 0, 0)],$$
where $\tG$ is the unreduced suspension of $G$, considered in Example~\ref{ex:G}. We computed that $k(\tG) = 1$, and we know from \eqref{eq:kwedge} that $k$ is unchanged by wedging with a free space. Therefore, we have $\kappa(\Sigma(2,3,12n-1)) = 2.$ 

The spectrum class of $-\Sigma(2,3,12n-1)$ is dual to that of $\Sigma(2,3,12n-1)$; compare the proof of Proposition~\ref{prop:duals}. We know from \cite[Example 2.14]{swfh} that $\tG$ is $\H$-dual to the space $\tilde T$ from Example~\ref{ex:torus}. Furthermore, $G_+$ is stably $(\R^{\dim G})$-dual to itself by the Wirthm\"uller isomorphism.\footnote{The Wirthm\"uller isomorphism \cite{Wirthmuller2, LMS} is usually formulated in equivariant stable homotopy theory built on a complete universe; that is, by allowing suspensions by arbitrary representations of $G$. In our setting, we only use the representations $\R, \tR$ and $\H$. Nevertheless, what is essential is that we can embed $G$ in one of these representations, in our case $\H$. The Thom space of its normal bundle is then $\Sigma^{3\R} G_+$, which shows that $G_+$ and $\Sigma^{3\R} G_+$ are $\H$-dual. After suspending by $\R$, we get that $G_+$ and $\Sigma^{4\R} G_+ \cong \Sigma^{\H} G_+$ are $(\R \oplus \H)$-dual, which shows the Wirthm\"uller isomorphism explicitly.} Since $\Sigma^{\H} G_+ \simeq \Sigma^4 G_+$, we can write the dual of $\Sigma{G_+}$ as the formal de-suspension of $\Sigma^{2}G_+$ by $\H$. We deduce that:
$$ \S(-\Sigma(2,3,12n-1)) = [(\tilde T \vee \underbrace{\Sigma^2 G_+ \vee \dots \vee \Sigma^2 G_+}_{n-1}, 0, 1)].$$
In Example~\ref{ex:torus} we computed $k(\tilde T) =1$, so we find that $\kappa(-\Sigma(2,3,12n-1))= 2 \cdot (1-1) = 0.$

The case of $\Sigma(2,3,12n-5)$ is similar to $\Sigma(2,3,12n-1)$, except now the reducible is in degree $-2$ and the irreducibles in degree $-1$. Thus, $\S(\Sigma(2,3,12n-5))$ is a formal de-suspension of $\S(\Sigma(2,3,12n-1))$ by $1/2$ copies of the representation $\H$. Therefore,
$$ \S(\Sigma(2,3,12n-5)) = [(\tilde G \vee \underbrace{\Sigma G_+ \vee \dots \vee \Sigma G_+,}_{n-1} 0, 1/2)].$$
The spectrum class for $-\Sigma(2,3,12n-5)$ is the dual of $\S(\Sigma(2,3,12n-1))$, and the formal suspension of $\S(-\Sigma(2,3,12n-1))$ by $1/2$ copies of $\H$. Thus,
$$ \S(-\Sigma(2,3,12n-5)) = [(\tilde T \vee \underbrace{\Sigma^2 G_+ \vee \dots \vee \Sigma^2 G_+}_{n-1}, 0, 1/2)].$$
From here we deduce that $\kappa(\Sigma(2,3,12n-5)) = \kappa(-\Sigma(2,3,12n-5)) = 1.$

Next, consider the Seiberg-\!Witten flow for $\Sigma(2,3,12n+1)$. This has one reducible in degree $0$ and $2n$ irreducibles in degree $-1$, coming in $k$ pairs related by the action of $j$. The attaching maps have to be trivial for homotopical reasons. We get:
\begin{equation}
\label{eq:s13}
 \S(\Sigma(2,3,12n+1)) = [(S^0 \vee \underbrace{\Sigma^{-1} G_+ \vee \dots \vee \Sigma^{-1} G_+}_n, 0, 0)].
 \end{equation}
Strictly speaking, by this we mean the spectrum class of 
$$(\H^+ \vee \underbrace{\Sigma^{3} G_+ \vee \dots \vee \Sigma^{3} G_+}_n, 0, 1)],$$
but we write it as in \eqref{eq:s13} for simplicity. Its dual is:
$$ \S(-\Sigma(2,3,12n+1)) = [(S^0 \vee \underbrace{G_+ \vee \dots \vee G_+}_n, 0, 0)].$$
We obtain $\kappa(\Sigma(2,3,12n+1)) = \kappa(-\Sigma(2,3,12n+1)) =0.$

Finally, the Seiberg-\!Witten flow for $\Sigma(2,3,12n+5)$ is analogous to that for $\Sigma(2,3,12n+1)$, except for an upward shift in dimension by $2$. Therefore, 
$$ \S(\Sigma(2,3,12n+5)) = [(S^0 \vee \underbrace{\Sigma^{-1} G_+ \vee \dots \vee \Sigma^{-1} G_+}_n, 0, -1/2)],$$
with dual
$$ \S(-\Sigma(2,3,12n+5)) = [(S^0 \vee \underbrace{G_+ \vee \dots \vee G_+}_n, 0, 1/2)].$$
We deduce that $\kappa(\Sigma(2,3,12n+5)) = 1$ and $\kappa(-\Sigma(2,3,12n+5)) =-1.$ This completes the proof of Theorem~\ref{thm:brieskorn}.

\subsection{Explicit bounds}
\label{sec:bounds}
For an integral homology sphere $Y$, define:
$$ \xi(Y) = \max \{p-q \mid p, q \in \Z, q > 1, \exists \ X^4 \text{ spin}, \del X = Y, \ Q(X) \equiv p (-E_8) \oplus q\hyp \},$$
where $Q(X)$ denotes the intersection form of $X$. 

The simplest way of obtaining an upper bound on $\xi(Y)$ is to find a compact spin $4$-manifold $X'$ with $\del X' = -Y$, and then apply Furuta's 10/8 theorem to $X \cup_Y X'$. If $X'$ has intersection form $p' (-E_8) \oplus q'\hyp$, from \eqref{eq:10_8} we get:
\begin{equation}
\label{eq:direct}
\xi(Y) \leq q'-p'-1.
\end{equation}
In particular, for $Y=S^3$, by taking $X'$ to be a four-ball we get $\xi(S^3) \leq -1$. Since the $K3$ surface has intersection form $2(-E_8) \oplus 3 \hyp$, we see that $$\xi(S^3)=-1.$$

A more refined way of getting upper bounds on $\xi(Y)$ is to find a compact, spin $4$-dimensional {\it orbifold}\footnote{Orbifolds were first introduced by Satake \cite{Satake} under the name of V-manifolds. The term V-manifold is used in some of the literature, e.g., in \cite{FukumotoFuruta} and \cite{SavelievFF}.} $X'$ with $\del X' = -Y$. Let $\t$ denote the spin structure on $X'$. Let also $X$ be a spin manifold with boundary $Y$ and intersection form $p(-E_8) \oplus q \hyp$, such that $q > 0$, as in the definition of $\xi$. Fukumoto and Furuta \cite{FukumotoFuruta} proved an analogue of the 10/8-theorem for closed, spin orbifolds. Applying it to $X \cup_Y X'$, it reads
\begin{equation}
\label{eq:ff}
b_2^+(X \cup_Y X') \geq 1 + \ind_{\C} \Dirac(X \cup_Y X').
\end{equation}
In \cite{FukumotoFuruta}, this is stated under the assumption $\ind_{\C} \Dirac(X \cup_Y X') > 0$. However, since $b_2^+(X \cup_Y X') = q + b_2^+(X') \geq q \geq 1,$ the inequality \eqref{eq:ff} remains true if $\ind_{\C} \Dirac(X \cup_Y X') \leq 0$.

Fukumoto and Furuta defined an invariant
$$ w(-Y, X', \t) =  \ind_{\C} \Dirac(X \cup_Y X') + \frac{1}{8}\sigma(X).$$
This turns out to be independent of $X$. When $X'$ is a plumbed spin orbifold, Saveliev \cite{SavelievFF} proved that $w(-Y, X', \t)$ coincides with the Neumann-Siebenmann invariant $-\bar \mu(-Y) = \bar \mu(Y)$ from \cite{NeumannPlumbed, SiebenmannPlumbed}. Thus, in this case, from \eqref{eq:ff} we obtain:
$$ q + b_2^+(X') \geq 1 + \bar \mu(Y) + p.$$
In particular, if $Y$ is a Seifert fibered homology sphere $\Sigma(a_1, \dots, a_k)$ with at least one of the $a_i$ even, we can take $X'$ to be the orbifold $D^2$-bundle over $S^2(a_1, \dots, a_k)$ associated to the Seifert fibration; we choose the orientation of $X'$ so that $\del X' = -Y$. Then $X'$ has a unique spin structure $\t$, and we have $b_2^+(X') = 1, b_2^-(X') = 0$; compare \cite{FukumotoFurutaUe, FukumotoBG}. We get the bound:
\begin{equation}
\label{eq:xiS}
\xi(\Sigma(a_1, \dots, a_k)) \leq - \bar \mu(\Sigma(a_1, \dots, a_k)).
\end{equation}
Applying the same reasoning to $-Y$ and $-X'$ instead of $Y$ and $X'$, since $b_2^+(-X')=0$, we get the bound:
\begin{equation}
\label{eq:xiSm}
\xi(-\Sigma(a_1, \dots, a_k)) \leq \bar \mu(\Sigma(a_1, \dots, a_k)) - 1.
\end{equation}

The $\bar \mu$ invariant for $\Sigma(a_1, \dots, a_k)$ can be computed explicitly; see \cite{NeumannPlumbed, NeumannRaymond}. In particular, for the Brieskorn spheres $\pm \Sigma(2,3,m)$ with $\gcd(m,6)=0$, from \eqref{eq:xiS} and \eqref{eq:xiSm} we get the concrete inequalities:
\begin{align}
\label{eq:Bound11}
\xi(\Sigma(2,3,12n-1)) &\leq 0, \hskip1cm \xi(-\Sigma(2,3,12n-1)) \leq -1, \\
\label{eq:Bound7}
\xi(\Sigma(2,3,12n-5)) &\leq -1, \hskip.72cm \xi(-\Sigma(2,3,12n-5)) \leq 0, \\
\label{eq:Bound13}
\xi(\Sigma(2,3,12n+1)) &\leq 0, \hskip1cm \xi(-\Sigma(2,3,12n+1)) \leq -1,\\
\label{eq:Bound5}
\xi(\Sigma(2,3,12n+5)) &\leq 1, \hskip1cm \xi(-\Sigma(2,3,12n+5)) \leq -2.
\end{align}

This paper provides a new method for obtaining bounds on $\xi$. Indeed, by Corollary~\ref{cor:1bdry}, we have:
\begin{equation}
\label{eq:xikappa}
\xi(Y) \leq \kappa(Y)-1.
\end{equation}
Given the values of $\kappa$ for $\pm \Sigma(2,3,m)$ computed in Theorem~\ref{thm:brieskorn}, we find that \eqref{eq:xikappa} gives the following bounds:
\begin{align}
\label{eq:bound11}
\xi(\Sigma(2,3,12n-1)) &\leq 1, \hskip1cm \xi(-\Sigma(2,3,12n-1)) \leq -1, \\
\label{eq:bound7}
\xi(\Sigma(2,3,12n-5)) &\leq 0, \hskip1cm \xi(-\Sigma(2,3,12n-5)) \leq 0, \\
\label{eq:bound13}
\xi(\Sigma(2,3,12n+1)) &\leq -1, \hskip.72cm \xi(-\Sigma(2,3,12n+1)) \leq -1,\\
\label{eq:bound5}
\xi(\Sigma(2,3,12n+5)) &\leq 0, \hskip1cm \xi(-\Sigma(2,3,12n+5)) \leq -2.
\end{align}

Comparing these with \eqref{eq:Bound11}-\eqref{eq:Bound5}, we see that $\kappa$ gives better bounds in two of the eight cases: namely, for $\xi(\Sigma(2,3,12n+1))$ and $\xi(\Sigma(2,3,12n+5))$. 

Let us see to what extent the information we get from \eqref{eq:direct},  \eqref{eq:Bound11}-\eqref{eq:Bound5} and \eqref{eq:bound11}-\eqref{eq:bound5} allows us to calculate $\xi(\pm \Sigma(2,3,m)).$ We do a case-by-case analysis.

\medskip
\noindent $\mathbf{Y=\pm\Sigma(2,3,12n-1)}.$ The Brieskorn sphere $-\Sigma(2,3,12n-1)$ is the boundary of the nucleus $N(2n)$ inside the elliptic surface $E(2n)$. The nucleus can be represented by the Kirby diagram
\begin{center}
 \begin{picture}(0,0)%
\includegraphics{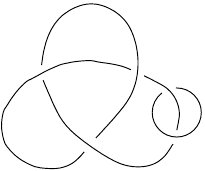}%
\end{picture}%
\setlength{\unitlength}{2368sp}%
\begingroup\makeatletter\ifx\SetFigFont\undefined%
\gdef\SetFigFont#1#2#3#4#5{%
  \reset@font\fontsize{#1}{#2pt}%
  \fontfamily{#3}\fontseries{#4}\fontshape{#5}%
  \selectfont}%
\fi\endgroup%
\begin{picture}(1625,1342)(2715,-2483)
\put(2971,-1325){\makebox(0,0)[lb]{\smash{{\SetFigFont{8}{9.6}{\rmdefault}{\mddefault}{\updefault}{\color[rgb]{0,0,0}$0$}%
}}}}
\put(4325,-1951){\makebox(0,0)[lb]{\smash{{\SetFigFont{8}{9.6}{\rmdefault}{\mddefault}{\updefault}{\color[rgb]{0,0,0}$-2n$}%
}}}}
\end{picture}%

 \end{center}
The intersection form of $N(2n)$ is equivalent to $\hyp$. By reversing the orientation of $N(n)$, we obtain a manifold with boundary $\Sigma(2,3,12n-1)$ and intersection form $\hyp$. From the definition of $\xi$, we get:
$$ -1 \leq \xi(\pm \Sigma(2,3,12n-1)).$$
In conjunction with \eqref{eq:Bound11}, we obtain:
$$ \xi(\Sigma(2,3,12n-1)) \in \{0, -1\}, \ \ \ \  \xi(-\Sigma(2,3,12n-1)) = -1.$$
We do not know the value of $\xi(\Sigma(2,3,12n-1))$ in general. However, for $n=1$, the complement of $N(2n)$ in the $K3$ surface has intersection form $2(-E_8) \oplus 2\hyp$. Therefore,
$$ \xi(\Sigma(2,3,11)) = 0.$$

\medskip
\noindent $\mathbf{Y=\pm\Sigma(2,3,12n-5)}.$ The manifold $-\Sigma(2,3,12n-5)$ is the boundary of the following plumbing of spheres:
\begin{center}
\begin{picture}(0,0)%
\includegraphics{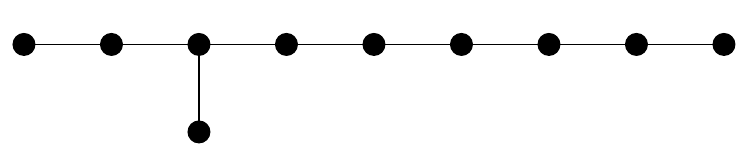}%
\end{picture}%
\setlength{\unitlength}{2763sp}%
\begingroup\makeatletter\ifx\SetFigFont\undefined%
\gdef\SetFigFont#1#2#3#4#5{%
  \reset@font\fontsize{#1}{#2pt}%
  \fontfamily{#3}\fontseries{#4}\fontshape{#5}%
  \selectfont}%
\fi\endgroup%
\begin{picture}(5048,1111)(-164,-1475)
\put(3451,-511){\makebox(0,0)[lb]{\smash{{\SetFigFont{8}{9.6}{\rmdefault}{\mddefault}{\updefault}{\color[rgb]{0,0,0}$-2$}%
}}}}
\put(4051,-511){\makebox(0,0)[lb]{\smash{{\SetFigFont{8}{9.6}{\rmdefault}{\mddefault}{\updefault}{\color[rgb]{0,0,0}$-2$}%
}}}}
\put(4651,-511){\makebox(0,0)[lb]{\smash{{\SetFigFont{8}{9.6}{\rmdefault}{\mddefault}{\updefault}{\color[rgb]{0,0,0}$-2n$}%
}}}}
\put(1351,-1411){\makebox(0,0)[lb]{\smash{{\SetFigFont{8}{9.6}{\rmdefault}{\mddefault}{\updefault}{\color[rgb]{0,0,0}$-2$}%
}}}}
\put(-149,-511){\makebox(0,0)[lb]{\smash{{\SetFigFont{8}{9.6}{\rmdefault}{\mddefault}{\updefault}{\color[rgb]{0,0,0}$-2$}%
}}}}
\put(451,-511){\makebox(0,0)[lb]{\smash{{\SetFigFont{8}{9.6}{\rmdefault}{\mddefault}{\updefault}{\color[rgb]{0,0,0}$-2$}%
}}}}
\put(1051,-511){\makebox(0,0)[lb]{\smash{{\SetFigFont{8}{9.6}{\rmdefault}{\mddefault}{\updefault}{\color[rgb]{0,0,0}$-2$}%
}}}}
\put(1651,-511){\makebox(0,0)[lb]{\smash{{\SetFigFont{8}{9.6}{\rmdefault}{\mddefault}{\updefault}{\color[rgb]{0,0,0}$-2$}%
}}}}
\put(2251,-511){\makebox(0,0)[lb]{\smash{{\SetFigFont{8}{9.6}{\rmdefault}{\mddefault}{\updefault}{\color[rgb]{0,0,0}$-2$}%
}}}}
\put(2851,-511){\makebox(0,0)[lb]{\smash{{\SetFigFont{8}{9.6}{\rmdefault}{\mddefault}{\updefault}{\color[rgb]{0,0,0}$-2$}%
}}}}
\end{picture}%

\end{center}
This plumbing has intersection form $(-E_8) \oplus \hyp$. If we reverse its orientation, we obtain a manifold with intersection form $E_8 \oplus \hyp$ and boundary $\Sigma(2,3,12n-5)$. We deduce that:
$$ -2 \leq \xi(\Sigma(2,3,12n-5)), \ \ \ \ 0 \leq \xi(-\Sigma(2,3,12n-5)).$$
In view of \eqref{eq:Bound7}, we get:
$$ \xi(\Sigma(2,3,12n-5)) \in \{-2, -1\}, \ \ \ \ \xi(-\Sigma(2,3,12n-5)) = 0.$$
For $n=1$, observe that the complement of the $(-E_{10})$-plumbing inside the $K3$ surface has intersection form $(-E_8) \oplus 2\hyp$. Therefore,
$$  \xi(\Sigma(2,3,7)) = -1.$$

\medskip
\noindent $\mathbf{Y=\pm\Sigma(2,3,12n+1)}.$ The Brieskorn sphere $-\Sigma(2,3,12n+1)$ is the boundary of the manifold
\begin{center}
 \begin{picture}(0,0)%
\includegraphics{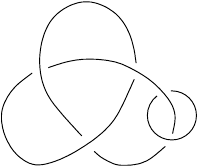}%
\end{picture}%
\setlength{\unitlength}{2368sp}%
\begingroup\makeatletter\ifx\SetFigFont\undefined%
\gdef\SetFigFont#1#2#3#4#5{%
  \reset@font\fontsize{#1}{#2pt}%
  \fontfamily{#3}\fontseries{#4}\fontshape{#5}%
  \selectfont}%
\fi\endgroup%
\begin{picture}(1580,1332)(1354,-2473)
\put(2905,-1907){\makebox(0,0)[lb]{\smash{{\SetFigFont{8}{9.6}{\rmdefault}{\mddefault}{\updefault}{\color[rgb]{0,0,0}$-2n$}%
}}}}
\put(1695,-1342){\makebox(0,0)[rb]{\smash{{\SetFigFont{8}{9.6}{\rmdefault}{\mddefault}{\updefault}{\color[rgb]{0,0,0}$0$}%
}}}}
\end{picture}%

 \end{center}
with intersection form $\hyp$. Therefore, we have:
$$ -1 \leq \xi(\pm \Sigma(2,3,12n+1)).$$
Moreover, when $n=1$ or $2$, the manifolds $\Sigma(2,3,13)$ and $\Sigma(2,3,25)$ bound homology balls, so, by applying \eqref{eq:direct}, we get:
$$ \xi(\pm \Sigma(2,3,13)) = \xi(\pm \Sigma(2,3,25)) =-1.$$
The inequalities in \eqref{eq:bound13} now give the answers for all $n$:
$$ \xi(\pm \Sigma(2,3,12n+1)) = -1.$$
Note that the result for $+\Sigma(2,3,12n+1)$ was not accessible from \eqref{eq:Bound13}. This provides a first example where $\kappa$ gives a better bound than the one from the filling method.

\medskip
\noindent $\mathbf{Y=\pm\Sigma(2,3,12n+5)}.$
The manifold $\Sigma(2,3,12n+5))$ is the boundary of the plumbing
\begin{center}
\begin{picture}(0,0)%
\includegraphics{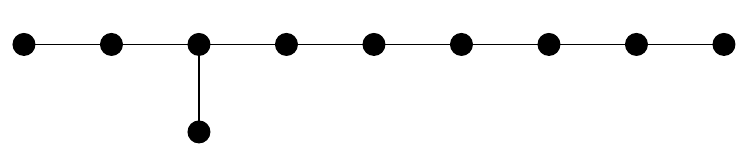}%
\end{picture}%
\setlength{\unitlength}{2763sp}%
\begingroup\makeatletter\ifx\SetFigFont\undefined%
\gdef\SetFigFont#1#2#3#4#5{%
  \reset@font\fontsize{#1}{#2pt}%
  \fontfamily{#3}\fontseries{#4}\fontshape{#5}%
  \selectfont}%
\fi\endgroup%
\begin{picture}(5048,1111)(-164,-1475)
\put(3451,-511){\makebox(0,0)[lb]{\smash{{\SetFigFont{8}{9.6}{\rmdefault}{\mddefault}{\updefault}{\color[rgb]{0,0,0}$-2$}%
}}}}
\put(4051,-511){\makebox(0,0)[lb]{\smash{{\SetFigFont{8}{9.6}{\rmdefault}{\mddefault}{\updefault}{\color[rgb]{0,0,0}$-2$}%
}}}}
\put(4651,-511){\makebox(0,0)[lb]{\smash{{\SetFigFont{8}{9.6}{\rmdefault}{\mddefault}{\updefault}{\color[rgb]{0,0,0}$2n$}%
}}}}
\put(1351,-1411){\makebox(0,0)[lb]{\smash{{\SetFigFont{8}{9.6}{\rmdefault}{\mddefault}{\updefault}{\color[rgb]{0,0,0}$-2$}%
}}}}
\put(-149,-511){\makebox(0,0)[lb]{\smash{{\SetFigFont{8}{9.6}{\rmdefault}{\mddefault}{\updefault}{\color[rgb]{0,0,0}$-2$}%
}}}}
\put(451,-511){\makebox(0,0)[lb]{\smash{{\SetFigFont{8}{9.6}{\rmdefault}{\mddefault}{\updefault}{\color[rgb]{0,0,0}$-2$}%
}}}}
\put(1051,-511){\makebox(0,0)[lb]{\smash{{\SetFigFont{8}{9.6}{\rmdefault}{\mddefault}{\updefault}{\color[rgb]{0,0,0}$-2$}%
}}}}
\put(1651,-511){\makebox(0,0)[lb]{\smash{{\SetFigFont{8}{9.6}{\rmdefault}{\mddefault}{\updefault}{\color[rgb]{0,0,0}$-2$}%
}}}}
\put(2251,-511){\makebox(0,0)[lb]{\smash{{\SetFigFont{8}{9.6}{\rmdefault}{\mddefault}{\updefault}{\color[rgb]{0,0,0}$-2$}%
}}}}
\put(2851,-511){\makebox(0,0)[lb]{\smash{{\SetFigFont{8}{9.6}{\rmdefault}{\mddefault}{\updefault}{\color[rgb]{0,0,0}$-2$}%
}}}}
\end{picture}%

\end{center}
with intersection form $(-E_8) \oplus \hyp$. By analogy with the case $\pm \Sigma(2,3,12n-5)$, we obtain
$$ 0 \leq \xi(\Sigma(2,3,12n+5)), \ \ \ \ -2 \leq \xi(-\Sigma(2,3,12n+5)).$$
The right hand side of \eqref{eq:Bound5} shows that:
$$ \xi(-\Sigma(2,3,12n+5))= -2.$$
On the other hand, to obtain the answer for $\xi(\Sigma(2,3,12n+5))$, we need the new bound \eqref{eq:bound5}, which gives:
$$ \xi(\Sigma(2,3,12n+5)) = 0.$$
When $n=1$, this could have also been seen by applying the inequality \eqref{eq:direct} to the positive definite $E_8$ plumbing with boundary $-\Sigma(2,3,12n+5)$.

\begin{remark}
An invariant similar to $\xi$ was considered by Bohr and Lee in \cite{BohrLee}: 
$$ m(Y) = \max \{ \tfrac{5}{4} \sigma(X) - b_2(X)  \mid X^4 \text{ spin}, \del X = Y \}.$$
This invariant was used in \cite{BohrLee} to study $\Z/2$-homology cobordism. We have:
$$ m(-Y)/2 = \max \{p-q \mid  p, q \in \Z, \exists \ X^4 \text{ spin}, \del X = Y, \ Q(X) \equiv p (-E_8) \oplus q\hyp \}.$$
Note that, unlike in the definition of $\xi$, here we do not assume that $q > 0$. Nevertheless, by taking  a connected sum with $S^2 \times S^2$, we obtain the bound $m(-Y)/2 \leq \xi(Y) + 1 \leq \kappa(Y)$. 
\end{remark}

\section{Relation to homological invariants}
\label{sec:alphas}

In this section we explore the relationship between the invariant $\kappa$ (constructed using equivariant K-theory) and the invariant $\alpha$ constructed in \cite{swfh} using equivariant (Borel) cohomology with $\Z/2$ coefficients. In the process we define yet another invariant of homology spheres, $\alpha_{\Q}$; this is constructed using equivariant cohomology with $\Q$ coefficients.

\subsection{The Borel homology of spaces of type SWF}
Let $X$ be a space of type $\swf$ at an even level $2t$. In \cite[Section 2.3]{swfh}, we associated to $X$ three quantities $a(X), b(X), c(X) \in \Z$. This can be done by considering either the Borel homology or the Borel cohomology of $X$. Let us start by reviewing the definition using Borel homology. 

Let $\F = \Z/2$ be the field with two elements. The reduced Borel homology $\tH_*^G(X; \F)$ is a module over the ring 
$$H^*(BG; \F) = \F[q, v]/(q^3),$$
where $q$ is in degree $1$ and $v$ in degree $4$. (Hence, $q$ and $v$ act on homology by lowering degrees by $1$ and $4$, respectively.) Consider the long exact sequence
$$ \dots \to \tH_*^G(X^{S^1}; \F) \to \tH_*^G(X; \F) \to \tH_*^G(X/X^{S^1}; \F) \to \cdots$$
Since the quotient $X/X^{S^1}$ has free $G$-action away from the basepoint, its homology is finite dimensional over $\F$. Therefore, in large enough degrees, the Borel homology  $\tH_*^G(X; \F)$ looks like that of the fixed point set $X^{S^1} \sim (\tC^t)^+$, which in turn is just isomorphic to $H_*(BG; \F).$ We find that  $\tH_*^G(X; \F)$ has an infinite ``tail'' of the form
\[ \xymatrixcolsep{.7pc}
\xymatrix{
\dots & \F  &  \F \ar@/_1pc/[l]_{q} &  \F \ar@/_1pc/[l]_{q} & 0 & \F \ar@/^1pc/[llll]^{v} & \F \ar@/_1pc/[l]_{q} \ar@/^1pc/[llll]^{v} & \F \ar@/_1pc/[l]_{q} \ar@/^1pc/[llll]^{v} & 0 & \dots  \ar@/^1pc/[llll]^{v} & \dots \ar@/^1pc/[llll]^{v} & \dots \ar@/^1pc/[llll]^{v}
} \]
Formally, the tail can be defined as the submodule
$$\iH_*^G(X; \F) := \bigcap_{l \geq 0} \im \bigl (v^l : \tH_{*+4l}^G(X; \F) \To \tH_*^G(X; \F) \bigr).$$
If we forget the action of $q$, the tail decomposes into three ``sub-tails,'' in degrees congruent to $2t, 2t+1$ and $2t+2$ mod $4$. We define $a(X), b(X), c(X)$ by asking that the minimal degrees of nonzero elements in each of the three sub-tails are $a(X), b(X)+1$ and $c(X)+2$, respectively. We will mostly be interested in the first quantity,
$$ a(X) = \min \{  r \equiv 2t \ (\mod 4) \mid \exists \ x,\ 0 \neq x \in  {\iH}_r^G(X; \F)  \}.$$

Let us now consider some variations of this, using Borel homology with coefficients in $\Z$ or $\Q$ rather than $\F$. Since $BG$ is an $\rp^2$-bundle over $\hp^{\infty}$ (see \cite[Section 2.1]{swfh}), we have
$$ H^*(BG; \Z) = \Z[s, v]/(s^2, 2s),$$
with $s$ in degree $2$ and $v$ in degree $4$. In large enough degrees $\tH_*^G(X; \Z)$ looks like the homology $H_*(BG; \Z)$, that is,
\begin{equation}
\label{eq:Ztail} \xymatrixcolsep{.7pc}
\xymatrix{
\dots & \Z  &  \Z/2  &  0 & 0 & \Z \ar@/^1pc/[llll]^{v} & \Z/2  \ar@/^1pc/[llll]^{v} & 0   & 0 & \dots  \ar@/^1pc/[llll]^{v} & \dots \ar@/^1pc/[llll]^{v} 
} \end{equation}
with $v$ being an isomorphism between the corresponding groups.

If we use $\Q$ coefficients, then $H^*(BG; \Q) = \Q[v]$ and   $\tH_*^G(X; \Q)$ has an infinite tail of the form
\[ \xymatrixcolsep{.7pc}
\xymatrix{
\dots & \Q  &  0  &  0 & 0 & \Q \ar@/^1pc/[llll]^{v} & 0  & 0   & 0 & \dots  \ar@/^1pc/[llll]^{v}  
} \]
For any Abelian group $A$ (in particular, for $A= \Z$ or $\Q$), we define
$$\iH_*^G(X; A) := \bigcap_{l \geq 0} \im \bigl (v^l : \tH_{*+4l}^G(X; A) \To \tH_*^G(X; A) \bigr).$$

Note that $\iH_*^G(X; \Q)$ is supported in degrees congruent to $2t$ mod $4$. We define an analogue of $a(X)$ using $\Q$ coefficients:
$$ a_{\Q}(X) = \min \{  r  \mid \exists \ x,\ 0 \neq x \in  {\iH}_r^G(X; \Q)  \}.$$

The relationship between $a$ and $a_{\Q}$ can be found via Borel homology over $\Z$, using the universal coefficients theorem. In simple cases we expect that $a = a_{\Q}$, but not so in general. If we inspect the sub-tail consisting in copies of $\Z$ in \eqref{eq:Ztail}, we observe two possible things that can ``go wrong'' towards the end of the sub-tail:

 First, the tail may end not in $\Z$ but in a torsion group. For example, the last $v$ map in the tail may be a projection $\Z \to \Z/2$. If so, the copy of $\Z/2$ survives in Borel homology with $\F$ coefficients, but not in Borel homology with $\Q$ coefficients, and we get $a(X) < a_{\Q}(X)$.

\begin{example}
\label{ex:kh2}
Consider the quaternionic Hopf fibration $S(\H) \hookrightarrow S(\H^2) \to \hp^1$. Pull back this $S(\H)$-bundle under a degree $2$ map from $\hp^1 \cong S^4$ to itself, and let $Z$ be the total space of the resulting bundle. The group $G \subset S(\H)$ acts freely on $Z$, and the quotient $Q = Z/G$ is an $\rp^2$-bundle on $S^4$. The classifying map $Q \to BG$ induces a map on homology, and, if we identify both $H_4(Q; \Z)$ and $H_4(BG; \Z)$ with $\Z$, this map in degree $4$ is given by multiplication by $\pm 2$. The unreduced suspension $\tZ$ of $Z$ is a space of type $\swf$ at level zero. There is a long exact sequence (compare \cite[Section 2.4]{swfh}):
$$ \dots \to H_*(Q; \Z) \to H_*(BG; \Z) \to \tH_*^G(\tZ; \Z) \to \cdots ,$$
from which we deduce that the sub-tail of $\iH^*_G(\tZ; \Z)$ in degrees divisible by $4$ ends with a $\Z/2$ in degree $4$. Consequently, we have $a(\tZ)=4$ but $a_{\Q}(\tZ)=8$.
\end{example}

Another thing that can happen at the end of the tail of $\Z$'s in \eqref{eq:Ztail} is that the last nonzero $v$ map ends in a $\Z$ summand of $\tH^*(X; \Z)$, but this last map has nontrivial cokernel. For example, suppose that the tail ends in a copy of $\Z$ in degree $r \equiv 2t \pmod 4$, with the last $v$ map having cokernel $\Z/2$. Then, the tail of Borel homology with $\Q$ coefficients ends in degree $r$ as well, but the tail with $\F$ coefficients ends in a higher degree $a(X) > a_{\Q}(X)=r$. This situation appears, for instance, for a space that is  equivariantly $\H^m$-dual (for some $m$) to the space $\tZ$ from Example~\ref{ex:kh2}. 

Nevertheless, in many cases neither of the above two anomalies appear. We have:

\begin{proposition}
\label{prop:aaQ}
Suppose that $X$ is a space of type $\swf$ at an even level $2t$, such that, for any $r \equiv 2t \! \pmod 4$:
\begin{enumerate}[(i)]
\item The group $\iH_r^G(X;\Z)$ has no $2$-torsion elements, and
\item There are no elements $x \in H_r^G(X; \Z)$ such that $0 \neq 2x \in {\iH}_r^G(X;\Z)$ but $x \not \in {\iH}_r^G(X; \Z)$.
\end{enumerate}
Then, we have $a(X) = a_\Q(X)$.
\end{proposition}
 
 \begin{proof}
 This is an application of the universal coefficients theorem.
 \end{proof}

Observe that the assumptions of Proposition~\ref{prop:aaQ} are satisfied for the spaces $\tG$ and $\tT$ considered in Example~\ref{ex:G} and Example~\ref{ex:torus}.

Let us now mention how the quantities $a$ and $a_{\Q}$ can be expressed in terms of Borel cohomology rather than Borel homology. When $A = \F$ or $\Q$, the Borel cohomology $\tH_G^*(X; A)$ has a tail similar to the one in Borel homology, except that the arrows increase degree. We get
$$ a(X) = \min \{  r \equiv 2t \ (\mod 4) \mid \exists \ x \in \tH^r_G(X; \F), \ v^l x \neq 0 \text{ for all } l \geq 0  \}$$
and
\begin{equation}
\label{eq:aq}
 a_{\Q}(X) = \min \{  r \mid \exists \ x \in \tH^r_G(X; \Q), \ v^l x \neq 0 \text{ for all } l \geq 0  \}.
 \end{equation}

\subsection{Equivariant K-theory and Borel cohomology}
\label{sec:kh}
Let us now explore the connection between $a_{\Q}(X)$ and the quantity $k(X)$ introduced in Definition~\ref{def:kx}. We will use the fact that, when we use $\Q$ coefficients, the Chern character gives an isomorphism between (non-equivariant) K-cohomology and the completion of ordinary cohomology.

Recall that $k(X)$ was defined in terms of the ideal $\i(X) \subset R(G)$, which is the image of the restriction map $\tK_*(X) \to \tK_*(X^{S^1}) \cong R(G)$. We also have an interpretation for $a_{\Q}(X)$ in terms of the ideal $\i(X)$:
\begin{proposition}
If $X$ is a space of type $\swf$ at an even level $2t$, then
$$ a_{\Q}(X) = 2t + 4\min \{k \geq 0 \mid   \exists \ \lambda \in \Z^*,  \mu \in \Z, \ \lambda z^k + \mu w \in \i(X) \}.$$
\end{proposition}

\begin{proof}
Let $F$ be the pointed space $(X/X^{S^1})/G$. The inclusion of $X^{S^1}$ into $X$ gives rise to a long exact sequence:
\begin{equation}
\label{eq:giraffe}
\dots \to \tH^*(F; \Q) \to \tH^*_G(X; \Q) \to \tH^*_G(X^{S^1}; \Q) \xrightarrow{f} \tH^{*+1}(F; \Q) \to \cdots
\end{equation}
Let us identify $\tH^*_G(X^{S^1}; \Q)$ with $\tH^{*-2t}(BG; \Q)$ using the equivalence $X^{S^1} \sim  (\tC^t)^+$. By \eqref{eq:aq} and exactness, we can write
 \begin{align*}
  a_{\Q}(X) &= \min \{ r \mid \exists \ x, 0 \neq x \in \tH^r(BG; \Q), f(x) = 0 \} \\
  &= 2t + 4\min \{k \mid f(v^k) \neq 0\}.
  \end{align*}
Similarly, we have a long exact sequence in equivariant K-theory:
 \begin{equation}
 \label{eq:zebra}
 \dots \to \tK(F) \otimes \Q \to \tK_G(X) \otimes \Q \to \tK_G(X^{S^1}) \otimes \Q \xrightarrow{g} \tK^{1}(F) \otimes \Q \to \cdots
 \end{equation}
and we can identify $\tK^*_G(X^{S^1}) \otimes \Q$ with $R(G) \otimes \Q$ by the Bott isomorphism. 

The maps $f$ and $g$ from \eqref{eq:giraffe} and \eqref{eq:zebra} are the compositions of the maps in the bottom, resp. top row of the commutative diagram:
$$\begin{CD}
R(G) \otimes \Q @>>> K(BG) \otimes \Q @>>> \tK^{1}(F)\otimes \Q\\
@. @V{\operatorname{ch}}V{\cong}V @V{\operatorname{ch}}V{\cong}V \\
H^*(BG; \Q) @>>> H^*(BG; \Q)^{\wedge}_{v} @>>> \tH^{\text{odd}}(F; \Q). 
\end{CD}$$
Here, the first maps in each row are given by completion: for $R(G)$ with respect to the augmentation ideal $\a=(w, z)$, and for the cohomology $H^*(BG; \Q) = \Q[v]$ with respect to the ideal $(v)$. Note that $w \in R(G)$ gets sent to zero under completion over $\Q$, so $K(BG) \otimes Q \cong R(G)^{\wedge}_{\a} \otimes \Q$ is the power series ring $\Q[[z]]$. The isomorphism in the second column (given by the Chern character) is the map $\Q[[z]] \to \Q[[v]], z \mapsto v$. 

From the diagram above we find an alternative expression for $a_{\Q}$ in terms of the top row:
$$ a_{\Q}(X) = 2t + 4\min \{ k \geq 0 \mid \exists \ \epsilon \in \Q, \ g(z^k + \epsilon w) = 0 \}.$$
The conclusion now follows from the exactness of \eqref{eq:zebra}.
\end{proof}

In view of this proposition, we can compare the quantities $a_{\Q}(X)$ and $k(X)$ simply by inspecting the ideal $\i(X)$. In particular, we have:

\begin{corollary}
\label{cor:kh}
Suppose that $X$ is a space of type $\swf$ at an even level $2t$, such that $\i(X)$ is of the form $(z^k)$ or $(w^k, z^k)$ for some $k \geq 0$. Then, 
$$a_{\Q}(X) = 2t + 4k(X) = 2t + 4k.$$
\end{corollary}



\subsection{Invariants of homology spheres} Let $Y$ be an integral homology sphere. (The whole discussion here can be extended to rational homology spheres with spin structures, but we restrict to integral homology spheres for simplicity.) In \cite[Section 3.5]{swfh}, we extracted from the $G$-equivariant Seiberg-\!Witten Floer homology of $Y$ three numerical invariants $\alpha(Y), \beta(Y), \gamma(Y) \in \Z$. Let us focus on $\alpha(Y)$, which can be expressed as
$$ \alpha(Y) = \tfrac{1}{2} \min \{ r \equiv 2\mu(Y) \! \pmod 4 \mid \exists x, \ 0\neq x \in {\iswfh}_r^G(Y; \F) \},$$
where $ {\iswfh}_r^G(Y; \F)$ is the ``infinite tail'' of $\swfh_r^G(Y; \F)$, and $\mu(Y)$ is the Rokhlin invariant. More concretely, if $g$ is a metric on $Y$ and $\nu \gg 0$ an eigenvalue cut-off as in Section~\ref{sec:fda}, we have:
$$ \alpha(Y) = (a(I_{\nu}) - \dim V^0_{-\nu})/2 - n(Y, g).$$
We can define a similar invariant using coefficients in $\Q$:
$$ \alpha_{\Q}(Y) = (a_{\Q}(I_{\nu}) - \dim V^0_{-\nu})/2 - n(Y, g).$$
Next, recall from Section~\ref{sec:k} that we have the Floer K-theoretic invariant
$$\kappa(Y) = 2k(I_{\nu}) - \bigl( \dim_{\R} V^0_{-\nu}(\H) \bigr) /2 - n(Y, g).$$

\begin{proposition}
\label{prop:Qswf}
Let $Y$ be a homology sphere.

$(a)$ Suppose that, for any $r \equiv 2\mu(Y) \! \pmod 4$, the group ${\iswfh}_r^G(X;\Z)$ has no $2$-torsion elements, and that there are no elements $x \in \swfh_r^G(X; \Z)$ such that $0 \neq 2x \in {\iswfh}_r^G(X;\Z)$ but $x \not \in {\iswfh}_r^G(X; \Z)$. Then, we have $\alpha(Y) = \alpha_\Q(Y)$.

$(b)$ Let $g$ be a metric on $Y$. If for all $\nu \gg 0$, either the ideal $\i(I_\nu)$ or $\i(\Sigma^{\tR} I_{\nu})$ (whichever is well-defined, depending on the parity of the level of the Conley index $I_{\nu}$) is of one of the types $(z^k)$ or $(w^k, z^k)$ for some $k \geq 0$, then $\alpha_{\Q}(X)= \kappa(X)$.
\end{proposition}

\begin{proof}
Part (a) follows from Proposition~\ref{prop:aaQ}. Part (b) follows from Corollary~\ref{cor:kh}, using the fact that the level of $I_{\nu}$ is $\dim V^0_{-\nu}(\tR)$.
\end{proof}

Note that all the examples considered in Sections~\ref{sec:psc} and ~\ref{sec:brieskorn} satisfy the hypotheses in both parts of Proposition~\ref{prop:Qswf}. Hence, for those manifolds $Y$ we have $\alpha(Y) = \alpha_{\Q}(Y) = \kappa(Y)$. We expect that this fails in more complicated examples. 

\bibliography{biblio}

\def\polhk#1{\setbox0=\hbox{#1}{\ooalign{\hidewidth
  \lower1.5ex\hbox{`}\hidewidth\crcr\unhbox0}}}
\begin{thebibliography}{10}
\providecommand{\url}[1]{\texttt{#1}}
\providecommand{\urlprefix}{URL }
\expandafter\ifx\csname urlstyle\endcsname\relax
  \providecommand{\doi}[1]{doi:\discretionary{}{}{}#1}\else
  \providecommand{\doi}{doi:\discretionary{}{}{}\begingroup
  \urlstyle{rm}\Url}\fi
\providecommand{\eprint}[2][]{\url{#2}}

\bibitem{AtiyahK}
M.~F. Atiyah, \emph{{$K$}-theory}, Lecture notes by D. W. Anderson, W. A.
  Benjamin, Inc., New York-Amsterdam, 1967.

\bibitem{AtiyahBP}
M.~F. Atiyah, \emph{Bott periodicity and the index of elliptic operators},
  Quart. J. Math. Oxford Ser. (2), \textbf{19}(1968), 113--140.

\bibitem{AtiyahSegal}
M.~F. Atiyah and G.~B. Segal, \emph{Equivariant {$K$}-theory and completion},
  J. Differential Geometry, \textbf{3}(1969), 1--18.

\bibitem{Bauer}
S.~Bauer, \emph{Intersection forms of spin four-manifolds}, e-print,
  \eprint{arXiv:1211.7092v1}.

\bibitem{BohrLee}
C.~Bohr and R.~Lee, \emph{Homology cobordism and classical knot invariants},
  Comment. Math. Helv., \textbf{77}(2002), no.~2, 363--382.

\bibitem{Bryan}
J.~Bryan, \emph{Seiberg-{W}itten theory and {${\bf Z}/2^p$} actions on spin
  {$4$}-manifolds}, Math. Res. Lett., \textbf{5}(1998), no. 1-2, 165--183.

\bibitem{Crabb}
M.~C. Crabb, \emph{Periodicity in {${\bf Z}/4$}-equivariant stable homotopy
  theory}, in \emph{Algebraic topology ({E}vanston, {IL}, 1988)}, Amer. Math.
  Soc., Providence, RI, volume~96 of \emph{Contemp. Math.}, pp. 109--124, 1989.

\bibitem{Donald}
A.~Donald, \emph{Embedding {S}eifert manifolds in ${S}^4$}, e-print,
  \eprint{arXiv:1203.6008}.

\bibitem{Donaldson}
S.~K. Donaldson, \emph{An application of gauge theory to four-dimensional
  topology}, J. Differential Geom., \textbf{18}(1983), no.~2, 279--315.

\bibitem{DonaldsonOr}
S.~K. Donaldson, \emph{The orientation of {Y}ang-{M}ills moduli spaces and
  {$4$}-manifold topology}, J. Differential Geom., \textbf{26}(1987), no.~3,
  397--428.

\bibitem{FreedmanTaylor}
M.~H. Freedman and L.~Taylor, \emph{{$\Lambda $}-splitting {$4$}-manifolds},
  Topology, \textbf{16}(1977), no.~2, 181--184.

\bibitem{Froyshov}
K.~A. Fr{\o}yshov, \emph{The {S}eiberg-{W}itten equations and four-manifolds
  with boundary}, Math. Res. Lett., \textbf{3}(1996), no.~3, 373--390.

\bibitem{FroyshovYM}
K.~A. Fr{\o}yshov, \emph{Equivariant aspects of {Y}ang-{M}ills {F}loer theory},
  Topology, \textbf{41}(2002), no.~3, 525--552.

\bibitem{FroyshovHM}
K.~A. Fr{\o}yshov, \emph{Monopole {F}loer homology for rational homology
  3-spheres}, Duke Math. J., \textbf{155}(2010), no.~3, 519--576.

\bibitem{FukumotoBG}
Y.~Fukumoto, \emph{The bounded genera and {$w$}-invariants}, Proc. Amer. Math.
  Soc., \textbf{137}(2009), no.~4, 1509--1517.

\bibitem{FukumotoFuruta}
Y.~Fukumoto and M.~Furuta, \emph{Homology 3-spheres bounding acyclic
  4-manifolds}, Math. Res. Lett., \textbf{7}(2000), no. 5-6, 757--766.

\bibitem{FukumotoFurutaUe}
Y.~Fukumoto, M.~Furuta, and M.~Ue, \emph{{$W$}-invariants and
  {N}eumann-{S}iebenmann invariants for {S}eifert homology {$3$}-spheres},
  Topology Appl., \textbf{116}(2001), no.~3, 333--369.

\bibitem{FurutaBrazil}
M.~Furuta, \emph{Some variants of {F}loer cohomology}, Mat. Contemp.,
  \textbf{2}(1992), 67--72, workshop on the Geometry and Topology of Gauge
  Fields (Campinas, 1991).

\bibitem{Furuta}
M.~Furuta, \emph{Monopole equation and the {$\frac{11}8$}-conjecture}, Math.
  Res. Lett., \textbf{8}(2001), no.~3, 279--291.

\bibitem{FurutaLi}
M.~Furuta and T.-J. Li, \emph{Intersection forms of spin 4-manifolds with
  boundary}, preprint (2013).

\bibitem{KMOS}
P.~B. Kronheimer, T.~S. Mrowka, P.~S. Ozsv{\'a}th, and Z.~Szab{\'o},
  \emph{Monopoles and lens space surgeries}, Ann. of Math. (2),
  \textbf{165}(2007), no.~2, 457--546.

\bibitem{LMS}
L.~G. Lewis, Jr., J.~P. May, M.~Steinberger, and J.~E. McClure,
  \emph{Equivariant stable homotopy theory}, volume 1213 of \emph{Lecture Notes
  in Mathematics}, Springer-Verlag, Berlin, with contributions by J. E.
  McClure, 1986.

\bibitem{swfh}
C.~Manolescu, \emph{Pin(2)-equivariant {S}eiberg-\!{W}itten {F}loer homology
  and the triangulation conjecture}, e-print, \eprint{arXiv:1303.2354v2}.

\bibitem{Spectrum}
C.~Manolescu, \emph{Seiberg-{W}itten-{F}loer stable homotopy type of
  three-manifolds with {$b_1=0$}}, Geom. Topol., \textbf{7}(2003), 889--932
  (electronic).

\bibitem{GluingBF}
C.~Manolescu, \emph{A gluing theorem for the relative {B}auer-{F}uruta
  invariants}, J. Differential Geom., \textbf{76}(2007), no.~1, 117--153.

\bibitem{Matsumoto}
Y.~Matsumoto, \emph{On the bounding genus of homology {$3$}-spheres}, J. Fac.
  Sci. Univ. Tokyo Sect. IA Math., \textbf{29}(1982), no.~2, 287--318.

\bibitem{MayBook}
J.~P. May, \emph{Equivariant homotopy and cohomology theory}, volume~91 of
  \emph{CBMS Regional Conference Series in Mathematics}, Published for the
  Conference Board of the Mathematical Sciences, Washington, DC, with
  contributions by M. Cole, G. Comeza{\~n}a, S. Costenoble, A. D. Elmendorf, J.
  P. C. Greenlees, L. G. Lewis, Jr., R. J. Piacenza, G. Triantafillou, and S.
  Waner, 1996.

\bibitem{MOY}
T.~Mrowka, P.~Ozsv{\'a}th, and B.~Yu, \emph{Seiberg-{W}itten monopoles on
  {S}eifert fibered spaces}, Comm. Anal. Geom., \textbf{5}(1997), no.~4,
  685--791.

\bibitem{NeumannPlumbed}
W.~D. Neumann, \emph{An invariant of plumbed homology spheres}, in
  \emph{Topology {S}ymposium, {S}iegen 1979 ({P}roc. {S}ympos., {U}niv.
  {S}iegen, {S}iegen, 1979)}, Springer, Berlin, volume 788 of \emph{Lecture
  Notes in Math.}, pp. 125--144, 1980.

\bibitem{NeumannRaymond}
W.~D. Neumann and F.~Raymond, \emph{Seifert manifolds, plumbing, {$\mu
  $}-invariant and orientation reversing maps}, in \emph{Algebraic and
  geometric topology ({P}roc. {S}ympos., {U}niv. {C}alifornia, {S}anta
  {B}arbara, {C}alif., 1977)}, Springer, Berlin, volume 664 of \emph{Lecture
  Notes in Math.}, pp. 163--196, 1978.

\bibitem{Nicolaescu}
L.~I. Nicolaescu, \emph{Finite energy {S}eiberg-{W}itten moduli spaces on
  4-manifolds bounding {S}eifert fibrations}, Comm. Anal. Geom.,
  \textbf{8}(2000), no.~5, 1027--1096.

\bibitem{Olum}
P.~Olum, \emph{Mappings of manifolds and the notion of degree}, Ann. of Math.
  (2), \textbf{58}(1953), 458--480.

\bibitem{AbsGraded}
P.~S. Ozsv{\'a}th and Z.~Szab{\'o}, \emph{Absolutely graded {F}loer homologies
  and intersection forms for four-manifolds with boundary}, Adv. Math.,
  \textbf{173}(2003), no.~2, 179--261.

\bibitem{Rokhlin}
V.~A. Rokhlin, \emph{New results in the theory of four-dimensional manifolds},
  Doklady Akad. Nauk SSSR (N.S.), \textbf{84}(1952), 221--224.

\bibitem{Satake}
I.~Satake, \emph{On a generalization of the notion of manifold}, Proc. Nat.
  Acad. Sci. U.S.A., \textbf{42}(1956), 359--363.

\bibitem{SavelievFF}
N.~Saveliev, \emph{Fukumoto-{F}uruta invariants of plumbed homology 3-spheres},
  Pacific J. Math., \textbf{205}(2002), no.~2, 465--490.

\bibitem{Schmidt}
B.~Schmidt, \emph{Spin $4$-manifolds and $\pin$-equivariant homotopy theory},
  {P}h. D. thesis, Universit{\"a}t Bielefeld (2003).

\bibitem{Segal}
G.~Segal, \emph{Equivariant {$K$}-theory}, Inst. Hautes \'Etudes Sci. Publ.
  Math., (1968), no.~34, 129--151.

\bibitem{SiebenmannPlumbed}
L.~Siebenmann, \emph{On vanishing of the {R}ohlin invariant and nonfinitely
  amphicheiral homology {$3$}-spheres}, in \emph{Topology {S}ymposium, {S}iegen
  1979 ({P}roc. {S}ympos., {U}niv. {S}iegen, {S}iegen, 1979)}, Springer,
  Berlin, volume 788 of \emph{Lecture Notes in Math.}, pp. 172--222, 1980.

\bibitem{Stolz}
S.~Stolz, \emph{The level of real projective spaces}, Comment. Math. Helv.,
  \textbf{64}(1989), no.~4, 661--674.

\bibitem{Wirthmuller2}
K.~Wirthm{\"u}ller, \emph{Equivariant {$S$}-duality}, Arch. Math. (Basel),
  \textbf{26}(1975), no.~4, 427--431.

\end{thebibliography}
\bibliographystyle{custom}

\end{document}